\documentclass[12pt, reqno]{amsart}
\usepackage{hyperref}
\usepackage{ amssymb }
\numberwithin{equation}{section}
\usepackage{hyperref}
\usepackage{caption}
\usepackage{graphicx}
\usepackage{float}
\usepackage{pgf,tikz}

\newcommand\link{\operatorname{link}}

\newcommand \SC{\operatorname{SC}}
\newcommand \CCG{\operatorname{CCG}}

\newcommand \ww{\operatorname{W}}
\newcommand \T{\operatorname{T}}
\newcommand \TT{\mathcal{T}}

\newcommand \hh{\mathcal{H}}

\usepackage{enumitem}

\newcommand \cp{\operatorname{CP}}
\newcommand \cc{\operatorname{CC}}
\newcommand \diam{\operatorname{diam}}

\newcommand \con{\mathcal{C}}

\newcommand \ind{\mathcal{I}}

\newcommand{\dd}{\operatorname{d}}
\newtheorem{theorem}{Theorem}[section]
\newtheorem{definition}[theorem]{Definition}
\newtheorem{cons}{Construction}
\newtheorem{lemma}[theorem]{Lemma}
\newtheorem{proposition}[theorem]{Proposition}
\newtheorem{example}[theorem]{Example}
\newtheorem{obs}[theorem]{Observation}

\newtheorem{remark}[theorem]{Remark}

\newtheorem{corollary}[theorem]{Corollary}
\newtheorem{setup}[theorem]{Set-up}

\newtheorem*{notation*}{Notation}
\newtheorem{notation}[theorem]{Notation}

\begin{document}

\title[Shellability of Higher Independence Complexes of Graphs]
{Shellability of Higher Independence Complexes of Graphs}

\dedicatory{Dedicated with deep gratitude to the memory of Professor J{\"u}rgen Herzog}
\author{Arka Ghosh}
\email{arkaghosh1208@gmail.com}

\author{S Selvaraja}
\email{selvas@iitbbs.ac.in}
\address{Department of Mathematics, IIT Bhubaneswar, Bhubaneswar, 752050, India}

\thanks{AMS Classification 2020: 05E45, 13F55,  05E40} 

\begin{abstract}
This paper investigates the shellability of $r$-independence complexes $\ind_r(G)$, a generalization of classical independence complexes introduced by Paolini and Salvetti. For a graph $G$, a subset $A \subseteq V(G)$ is $r$-independent if every connected component of the induced subgraph $G[A]$ has at most $r$ vertices. The associated simplicial complex $\ind_r(G)$ has been the subject of significant interest due to its connections to combinatorial topology and commutative algebra. We address the classification problem for shellable $r$-independence complexes, focusing on block graphs, trees, and related families. Our main results establish sufficient conditions for shellability based on structural graph parameters such as diameter and forbidden subgraphs. Furthermore, we develop constructive techniques for generating shellable complexes through graph operations, including star-clique attachments, clique whiskering, and clique cycle constructions. These results extend and refine earlier work on classical independence complexes and provide a framework for understanding the topological and algebraic properties of higher independence complexes in structured graph families.
\end{abstract}
\maketitle

\section{Introduction}

Let $G = (V(G), E(G))$ be a simple (no loops, no multiple edges) undirected graph with vertex set $V(G) = \{x_1,\dots, x_n\}$ and edge set $E(G)$. A subset $A \subseteq V(G)$ is called an \emph{independent set} if no two vertices in $A$ are adjacent, equivalently, if the induced subgraph $G[A]$ contains no edges. The \emph{independence complex} of $G$, denoted $\ind(G)$, is the simplicial complex whose faces are the independent sets of $G$, that is, 
$\ind(G) = \{ A \subseteq V(G) \mid A \text{ is independent in } G \}.$
Independence complexes have been studied extensively due to their rich combinatorial, topological, and algebraic structure \cite{ABM05, Ca07, Me03, Stanbook, vill_cohen, Wood2}. A central theme in this area is the classification of graph families whose independence complexes exhibit desirable shellability properties \cite{crupi, fv, VanVilla, vill_cohen, Wood2, Russ11}, and the construction of new graphs from given ones -- via operations such as attaching whiskers, adding cliques, or other modifications -- so that the resulting independence complexes are shellable \cite{BFH15, CookNagel, DochEng09, FH, Hibi_cameronwalker, VanVilla, vill_cohen, Wood2}.

The concept of shellability has a rich history dating back to the 19th century, where it first appeared implicitly in the study of convex polytopes.  A (not necessarily pure) simplicial complex $\Delta$ is said to be \emph{shellable} if its facets can be ordered $F_1, F_2, \dots, F_t$ such that for each $k > 1$, the simplicial complex $\langle F_k \rangle \cap \langle F_1, \dots, F_{k-1} \rangle$ is pure and $(\dim F_k - 1)$-dimensional, where $\langle \mathcal{F} \rangle$ denotes the simplicial complex generated by the family $\mathcal{F}$. Originally developed for pure complexes, the theory was significantly expanded in the 1990s when Bj\"orner, Lov\'asz, and Yao \cite{BLY1992} encountered nonpure complexes in complexity theory with topological properties resembling those of shellable complexes. This led Bj\"orner and Wachs \cite{BW96, BW97} to develop a comprehensive theory of shellability for nonpure complexes and posets.
Shellability extends beyond simplicial complexes to various combinatorial structures and plays a crucial role in algebraic combinatorics and commutative algebra \cite{Herzog'sBook, PRS98, Stanbook}. From an algebraic perspective, shellability provides a powerful combinatorial criterion for establishing the sequentially Cohen-Macaulay property of Stanley-Reisner rings. When a simplicial complex is shellable, its Stanley-Reisner ring is sequentially Cohen-Macaulay, with profound implications for both algebraic invariants and topological properties \cite{PRS98, Stanbook}. Topologically, shellability guarantees strong connectivity properties: every pure $d$-dimensional shellable complex is homotopy equivalent to a wedge of $d$-spheres and is $(d-1)$-connected. Despite these elegant theoretical properties, the computational complexity of shellability is challenging -- for every $d \geq 2$, determining whether a pure $d$-dimensional simplicial complex is shellable is NP-hard, hence NP-complete, \cite{GPPZTW19}.

Paolini and Salvetti \cite{PS18} introduced the \emph{$r$-independence complex} of a graph $G$, which generalizes the classical independence complex. A subset $A \subseteq V(G)$ is called \emph{$r$-independent} if every connected component of the induced subgraph $G[A]$ contains at most $r$ vertices. The corresponding simplicial complex, denoted $\ind_r(G)$, consists of all $r$-independent subsets of $V(G)$; in particular, $\ind_1(G) = \ind(G)$. These complexes have been studied in various contexts \cite{FPSAA23, AJM24, HJ15, KRK25, PSA22, DochEng09, TG06}, with the shellability of $r$-independence complexes remaining a central question.  
In this work, we address the shellability of $r$-independence complexes through the framework of chordal hypergraph techniques introduced by Woodroofe \cite{Russ11}. In particular, we rely on the notion of chordal hypergraphs, which provides a powerful tool for analyzing shellability in generalized independence complexes.

Given a graph $G$ and an integer $r \geq 1$, we define the hypergraph $\con_r(G)$ associated with $G$ as follows: the vertex set of $\con_r(G)$ is $V(\con_r(G)) = V(G)$, and the edge set consists of all subsets $S \subseteq V(G)$ of size $r+1$ such that the induced subgraph $G[S]$ is connected. Note that when $r = 1$, the hypergraph $\con_1(G)$ coincides with the original graph $G$.  
The independence complex of this hypergraph coincides with the $r$-independence complex of $G$, i.e., 
$\ind(\con_r(G)) = \ind_r(G)$ for all  $r \geq 1$ \text{\cite[Proposition 2.6]{FPSAA23}}.
Recently, \cite[Theorem 2.7]{FPSAA23} claimed that if $G$ is a tree, then $\con_r(G)$ is a chordal hypergraph for all $r \geq 2$, and consequently, by \cite[Corollary 5.4]{Russ11}, $\ind_r(G)$ is shellable. However, using Proposition \ref{not-chordal}, we construct a class of counterexamples that show this claim does not hold in general. Hence, the shellability of $\ind_r(G)$ for trees is not guaranteed and, to the best of our knowledge, no prior work has explored this direction.

We focus on \emph{block graphs}, a classical family of graphs introduced by Harary~\cite{Harary63}. Recall that a graph $G$ is a block graph if every \emph{block} -- that is, a maximal connected subgraph that remains connected upon the removal of any single vertex -- is a complete graph. In this work, we establish several results connecting the structure of block graphs (and related tree-like graphs) to the shellability of their $r$-independence complexes.

Our main contributions are summarized as follows. First, for any graph $G$ with $n \geq 3$ vertices, we show that if $n-2 \leq r \leq n-1$, then $\ind_r(G)$ is shellable (Theorem~\ref{chordal-cond}). We then generalize this result to forests: for any forest $G$, we prove that $\ind_r(G)$ is shellable for all $r \geq n - 5$ (Theorem~\ref{tree-lower}). Next, we establish that $\ind_2(G)$ is shellable for any block graph $G$ (Theorem~\ref{block-2}), and more generally, if a block graph has diameter at most $4$, then $\ind_r(G)$ is shellable for all $r \geq 2$ (Theorem~\ref{block-diam}). For trees specifically, we prove that if $\operatorname{diam}(G) \leq 5$, then $\ind_r(G)$ is shellable for all $r \geq 2$ (Theorem~\ref{tree-diam}).
We also explore shellability for larger independence parameters. If $G$ is a $\T_3$-free block graph, then $\ind_3(G)$ is shellable, and as a corollary, the same holds for any forest (Theorem~\ref{block:3} and Corollary~\ref{3:tree}). Extending these ideas to more general tree-like structures, we show that if $G$ is a $\mathcal{T}_2$-graph (a generalization of lobster trees; see Definition~\ref{def:T2-graph}), then $\ind_4(G)$ is shellable (Theorem~\ref{block-4}). Finally, if $G$ is a $\mathcal{T}_1$-graph (a generalization of caterpillar trees; see Definition~\ref{def:T1-graph}), then $\ind_r(G)$ is shellable for all $r \geq 5$ (Theorem~\ref{block-5}).

The next focus of this paper is to investigate how various graph modifications -- such as attaching whiskers, adding cliques, or performing other structural operations -- affect the shellability of higher independence complexes. A foundational result by Villarreal~\cite{vill_cohen} shows that for any graph $G$, the independence complex of its whiskered version $W(G)$ -- formed by attaching a pendant vertex (whisker) to each vertex of $G$ -- is Cohen–Macaulay. This was subsequently refined by Dochtermann and Engström~\cite{DochEng09} and independently by Woodroofe~\cite{Wood2}, who showed that $\ind(W(G))$ is vertex decomposable. 
Since vertex decomposability implies shellability, this strengthens Villarreal's result.
Building on this, Cook and Nagel~\cite{CookNagel} introduced the vertex clique-whiskered graph $G^\pi$ and proved that $\ind(G^\pi)$ is both unmixed and vertex decomposable. Biermann et al.~\cite{BFH15} further provided sufficient conditions on a subset $S \subseteq V(G)$ such that $\ind(G \cup W(S))$ is vertex decomposable, where only the vertices in $S$ are whiskered. 
% In a related direction, Hibi et al.~\cite{Hibi_cameronwalker} showed that attaching a complete graph to each vertex of $G$ produces a graph whose independence complex remains unmixed and vertex decomposable.

Motivated by these developments, we study the shellability of $r$-independence complexes $\ind_r(G)$ for a family of graphs constructed via star-clique attachments. Given a graph $H$ and a vertex cover $S \subseteq V(H)$, we define, for a fixed integer $r \geq 1$, the graph $\CCG(H, S, r)$ as the graph obtained by attaching a star-clique graph $\SC(x)$ with $|\SC(x)| \geq r+1$ to each vertex $x \in S$. Our main result in this setting is the following.

\vskip 1mm
\noindent
\textbf{Theorem~\ref{whisker}.}  
Let $G = \CCG(H, S, r)$, where $r \geq 1$. Then $\ind_r(G)$ is shellable.

\vskip 1mm
This result extends to a broader class of chordal graphs. Specifically, if $H$ is chordal and $G = \CCG(H, V(H), t)$ for some integer $t \geq 1$, then $\ind_r(G)$ is shellable whenever $r \leq 2t + 1$ (Theorem~\ref{she}). Moreover, we provide a counterexample demonstrating that this bound is sharp: the conclusion fails for $r > 2t + 1$ (Example~\ref{she-higher}).
We also generalize the \emph{clique whiskering} construction introduced by Cook and Nagel~\cite{CookNagel}. Given a graph $G$, a clique vertex-partition $\pi = \{W_1, \ldots, W_k\}$, and an integer $r \geq 1$, the \emph{$r$-clique whiskering} of $G$ with respect to $\pi$, denoted $G^\pi_r$, is obtained by adding $t_i \geq r$ new vertices to each clique $W_i$ and forming a clique on $W_i \cup \{x_{i,1}, \ldots, x_{i,t_i}\}$. We prove the following.

\vskip 1mm
\noindent
\textbf{Theorem~\ref{clique-whisker}.}  
Let $G$ be a graph, and let $\pi = \{W_1, \ldots, W_k\}$ be a clique vertex-partition of $G$. Then $\ind_r(G^\pi_r)$ is shellable for all $r \geq 1$.

\vskip 1mm
Finally, we introduce \emph{clique cycle graphs}, which generalize ordinary cycles by replacing edges with cliques. An \emph{$n$-clique cycle} $\mathrm{CC}(B_1, \ldots, B_n)$ consists of a cyclic arrangement of $n \geq 3$ cliques in which each consecutive pair shares exactly one vertex. We consider graphs $G(r)$ constructed from an $n$-clique cycle by attaching a star-clique graph $\SC(x)$ to each vertex in a subset $A$ of the connecting vertices, requiring that $|V(\SC(x))| \geq r+1$ for at least one $x \in A$. Our final main result is the following.

\vskip 1mm
\noindent
\textbf{Theorem~\ref{clique-cycle}.}  
Let $G = G(r)$ for some $r \geq 1$. Then $\ind_r(G)$ is shellable.

\vskip 1mm
To conclude, we provide an example showing that the conclusions of Theorems~\ref{whisker}, \ref{clique-whisker}, and \ref{clique-cycle} may fail for $\CCG(H, S, t)$, $G^\pi_t$, or $G(t)$ when $t < r$ (Example~\ref{last-ex}).

Our paper is structured as follows. In Section \ref{preliminaries}, we introduce the necessary terminology and foundational results that underpin our work. Section \ref{technical} presents key technical lemmas and propositions essential for establishing our main results. In Section \ref{con-chor}, we investigate the shellability of $\ind_r(G)$, providing sufficient conditions for various classes of graphs. Finally, in Section \ref{modifications}, we examine the shellability of $r$-independence complexes in graph constructions.

\section{Notation and Preliminaries}\label{preliminaries}

In this section, we establish the basic definitions and notation required for the main results.
Let $G$ be a finite simple graph without isolated vertices. For a graph $G$, let $V(G)$ and $E(G)$ denote the vertex set and edge set of $G$, respectively. The degree of a vertex $x \in V(G)$, denoted by $\deg_G(x)$ (or simply $\deg(x)$), is the number of edges incident to $x$.
A subgraph $H \subseteq G$ is called \emph{induced} if for any two vertices $u, v \in V(H)$, the edge $\{u, v\}$ belongs to $E(H)$ if and only if $\{u, v\} \in E(G)$. For a subset $A \subseteq V(G)$, the \emph{induced subgraph} of $G$ on $A$ is denoted by $G[A]$; that is, $G[A]$ is the subgraph of $G$ whose vertex set is $A$ and whose edge set consists of all edges of $G$ with both endpoints in $A$.
For a subset ${u_1, \ldots, u_r} \subseteq V(G)$, we define $N_G(u_1, \ldots, u_r) = \{v \in V(G) \mid \{u_i, v\} \in E(G) \text{ for some } 1 \leq i \leq r\}$ and $N_G[u_1, \ldots, u_r] = N_G(u_1, \ldots, u_r) \cup \{u_1, \ldots, u_r\}$.
For a subset $U \subseteq V(G)$, we denote by $G \setminus U$ the induced subgraph of $G$ on the vertex set $V(G) \setminus U$.

A \emph{path} in $G$ is a sequence of distinct vertices $v_0v_1\cdots v_k$ such that $\{v_i, v_{i+1}\} \in E(G)$ for all $0 \leq i \leq k-1$. The \emph{length} of a path is the number of edges it contains. A graph $G$ is called \emph{connected} if for any two vertices $u, v \in V(G)$, there exists a path from $u$ to $v$. A \emph{connected component} of $G$ is a maximal connected subgraph of $G$.
The \emph{distance} between two vertices $u$ and $v$ in $G$, denoted by $\dd(u, v)$, is the length of the shortest path connecting $u$ and $v$. The \emph{diameter} of $G$, denoted by $\diam(G)$, is the maximum distance between any pair of vertices in $G$. Formally,
$\diam(G) = \max_{u, v \in V(G)} \dd(u, v).$
A vertex $v \in V(G)$ is called a \emph{cut-vertex} if the removal of $v$ from $G$ increases the number of connected components of $G$. A vertex $v \in V(G)$ is called a \emph{simplicial vertex} if the subgraph induced by its neighbors is a complete graph (or clique); that is, for every pair of neighbors $u$ and $w$ of $v$, the edge $\{u, w\} \in E(G)$.

A \emph{simplicial complex} $\Delta$ on a vertex set $V = \{x_1, \ldots, x_n\}$ is a collection of subsets of $V$ satisfying two conditions: (i) every singleton $\{x_i\}$ belongs to $\Delta$ for each $x_i \in V$, and (ii) $\Delta$ is closed under taking subsets; that is, if $F \in \Delta$, then every subset $F' \subseteq F$ also belongs to $\Delta$. The elements of $\Delta$ are called \emph{faces}, and the maximal faces under inclusion are called \emph{facets}. 
Several important constructions are associated with a face $F \in \Delta$: the \emph{link} of $F$ is defined as $\link_\Delta(F) = \{ F' \subseteq V \mid F' \cap F = \emptyset \text{ and } F' \cup F \in \Delta \}$; the \emph{deletion} of $F$ is $\Delta \setminus F = \{ H \in \Delta \mid H \cap F = \emptyset \}$.

A \emph{hypergraph} $\mathcal{H}$ is a pair $(V(\mathcal{H}), E(\mathcal{H}))$, where $V(\mathcal{H})$ is the vertex set and $E(\mathcal{H}) \subseteq 2^{V(\mathcal{H})}$ is the edge set, consisting of subsets of $V(\mathcal{H})$ called the edges. A hypergraph is called \emph{simple} if no edge is contained in another edge; that is, for any distinct edges $e, f \in E(\mathcal{H})$, we have $e \not\subseteq f$. In this paper, we consider all hypergraphs to be simple. A subset $S \subseteq V(\mathcal{H})$ is called \emph{independent} if there is no $e \in E(\mathcal{H})$ such that $e \subseteq S$. The \emph{independence complex} of a hypergraph $\mathcal{H}$, denoted $\ind(\mathcal{H})$, is the simplicial complex whose simplices correspond to the independent subsets of $V(\mathcal{H})$.
The \emph{deletion} of the vertex $v$, denoted by $\mathcal{H} \setminus v$, is the hypergraph defined as:
$V(\mathcal{H} \setminus v) = V(\mathcal{H}) \setminus \{v\}$, 
$E(\mathcal{H} \setminus v) = \{e \in E(\mathcal{H}) \mid v \notin e\}.$
The \emph{contraction} of the vertex $v$, denoted by $\mathcal{H} / v$, is the hypergraph defined as:
\[
V(\mathcal{H} / v) = V(\mathcal{H}) \setminus \{v\}, \quad 
E(\mathcal{H} / v) = \min\bigl\{e \setminus \{v\} \mid e \in E(\mathcal{H})\bigr\},
\]
where $\min(S)$ denotes the inclusion-wise minimal subsets of a collection of sets $S$. Contractions and deletions of distinct vertices can be performed sequentially, and it is well known that the resulting hypergraph does not depend on the order of these operations \cite[Lemma 3]{Se75}.

A hypergraph $\mathcal{H}'$ obtained from $\mathcal{H}$ by deleting a subset of vertices $V_d \subseteq V(\mathcal{H})$ and contracting another subset $V_c \subseteq V(\mathcal{H})$, with $V_d \cap V_c = \emptyset$, is called a \emph{minor} of $\mathcal{H}$, denoted by $\mathcal{H} \setminus V_d / V_c$. If $V_d = \emptyset$, the hypergraph $\mathcal{H}'$ is called a \emph{$c$-minor} of $\mathcal{H}$. Furthermore, if a $c$-minor $\mathcal{H}'$ contains no singleton edges, it is called a \emph{$c'$-minor} of $\mathcal{H}$.

The concept of $\ww$-chordality\footnote{The notation honors Russ Woodroofe's foundational work.}, introduced by Woodroofe, generalizes the notion of chordality from graphs to hypergraphs and provides a framework for studying shellability and linear resolutions.

\begin{definition}[{\cite[Definition 4.2 and Definition 4.3]{Russ11}}]\mbox{}
    \begin{enumerate}
        \item A vertex $v \in V(\mathcal{H})$ in a hypergraph $\mathcal{H}$ is called a \emph{simplicial vertex} if for any two edges $e_1, e_2 \in E(\mathcal{H})$ containing $v$, there exists an edge $e_3 \in E(\mathcal{H})$ such that  
        $e_3 \subseteq (e_1 \cup e_2) \setminus \{v\}.$ 
        \item A hypergraph $\mathcal{H}$ is said to be \emph{$\ww$-chordal} if every minor of $\mathcal{H}$ contains a simplicial vertex.  
    \end{enumerate}
\end{definition}

These definitions lay the foundation for analyzing the shellability and sequentially Cohen--Macaulay properties of independence complexes.

\begin{theorem}[{\cite[Corollary 5.3 and Corollary 5.4]{Russ11}}] \label{russ-result} \mbox{}
    \begin{enumerate}
        \item Let $\mathcal{H}$ be a hypergraph such that every contraction of $\mathcal{H}$ has a simplicial vertex. Then $\ind(\mathcal{H})$ is shellable, and therefore the Stanley--Reisner ideal of $\ind(\mathcal{H})$ is sequentially Cohen--Macaulay.
        \item If $\mathcal{H}$ is a $\ww$-chordal hypergraph, then $\ind(\mathcal{H})$ is shellable, and therefore the Stanley--Reisner ideal of $\ind(\mathcal{H})$ is sequentially Cohen--Macaulay.
    \end{enumerate}
\end{theorem}

For any undefined terminology and further basic definitions and properties, we refer the reader to \cite{Herzog'sBook} and \cite{west}.

\section{Technical Preliminaries and Key Results}\label{technical}
In this section, we establish the technical groundwork necessary for the proof of our main theorem. We present a series of key lemmas and propositions that serve as the foundation for our arguments.
 
A collection $\mathcal{F}$ of graphs is called a \emph{hierarchy} if, for every nonempty graph $G \in \mathcal{F}$, the graph $G \setminus u$ also belongs to $\mathcal{F}$ for every vertex $u \in V(G)$.
The following proposition establishes a connection between the structural properties of a graph hierarchy and the $\ww$-chordality of their associated hypergraphs. In particular, it provides a sufficient condition for $\con_r(G)$ to be $\ww$-chordal, based on the existence of simplicial vertices in its minors.

\begin{proposition}\label{del-ope}
    Let $\mathcal{F}$ be a hierarchy of graphs and $r$ a positive integer. If every $c'$-minor of $\con_r(G)$ contains a simplicial vertex for any $G \in \mathcal{F}$, then $\con_r(G)$ is $\ww$-chordal.
\end{proposition}

\begin{proof}
    We first show that every $c$-minor of $\con_r(G)$ contains a simplicial vertex for any $G \in \mathcal{F}$. Let $\hh$ be a $c$-minor of $\con_r(G)$, i.e., $\hh = \con_r(G) / V_c$ for some $V_c \subseteq V(G)$. We consider two cases based on the presence of singleton edges in $\hh$:
  If $\hh$ has no singleton edges, then $\hh$ is a $c'$-minor of $\con_r(G)$. By assumption, $\hh$ contains a simplicial vertex.
 Suppose $\hh$ contains singleton edges. Let $S$ denote the set of all singleton edges in $\hh$. Consider the hypergraph
    $\hh \setminus S = (\con_r(G) / V_c) \setminus S.$
    By \cite[Lemma 3]{Se75}, we have
    $\hh \setminus S = (\con_r(G) \setminus S) / V_c.$
    Additionally, for any vertex $x \in V(G)$ and $r \geq 1$, we observe that
    $\con_r(G) \setminus x = \con_r(G \setminus x).$
    Extending this identity to the set $S$, we obtain
    $\hh \setminus S = \con_r(G \setminus S) / V_c.$
    Since $G \setminus S \in \mathcal{F}$ (as $\mathcal{F}$ is a hierarchy), and $\hh \setminus S$ is a $c'$-minor of $\con_r(G \setminus S)$, the hypothesis implies that $\hh \setminus S$ contains a simplicial vertex. Thus, $\hh$ must also contain a simplicial vertex.
In both cases, we conclude that every $c$-minor of $\con_r(G)$ contains a simplicial vertex. Since any minor can be obtained through deletions and contractions, and deletions correspond to graphs in the hierarchy by definition, it follows that all minors of $\con_r(G)$ possess a simplicial vertex. Hence, $\con_r(G)$ is $\ww$-chordal for every $G \in \mathcal{F}$.
\end{proof}

In \cite[Theorem~4.3]{HJ15}, Dao and Schweig proved that a vertex $v$ of a graph $G$ is simplicial if and only if it remains simplicial in $\con_r(G)$ for all $r \geq 1$. However, this equivalence does not hold for general graphs. For example, consider the cycle graph $C_4$: while no vertex in $C_4$ is simplicial, every vertex in $\con_r(C_4)$ becomes simplicial for $r = 2$ (or $r = 3$ under their construction). This demonstrates that the claimed equivalence fails in general.
Motivated by this observation, we examine the persistence of simplicial vertices in minors of the hypergraph $\con_r(G)$.
%While the following result is applicable to arbitrary graphs, we restrict our attention to chordal graphs due to their structural significance in our context.

\begin{lemma}\label{first-sim}
    Let $G$ be a graph, and let $\mathcal{H}$ be a minor of the hypergraph $\con_r(G)$ for some $r \geq 1$. If $x$ is a simplicial vertex of $G$ and $x \in V(\mathcal{H})$, then $x$ remains a simplicial vertex in $\mathcal{H}$.
\end{lemma}

\begin{proof}
    %If $x$ is an isolated vertex in $\mathcal{H}$, then by definition, $x$ is simplicial in $\mathcal{H}$. Assume now that $x$ is not isolated in $\mathcal{H}$. 
    Suppose first that $x$ belongs to exactly one edge of $\mathcal{H}$; then $x$ is simplicial by definition. Now consider the case where $x$ belongs to two distinct edges $e_1, e_2 \in E(\mathcal{H})$, with $e_1 \neq e_2$. By construction of the minor, there exist edges $f_1, f_2 \in E(\con_r(G))$ such that
    $e_i = f_i \cap V(\mathcal{H})$ for $i = 1, 2.$
    Since $x$ is simplicial in $G$, the induced subgraph of $G$ on $(f_1 \cup f_2) \setminus \{x\}$ is connected. Moreover, the fact that $f_1 \neq f_2$ implies that
    $|f_1 \cup f_2| \geq r + 2,$
    and hence,
    $|(f_1 \cup f_2) \setminus \{x\}| \geq r + 1.$
    Select a subset $f \subseteq (f_1 \cup f_2) \setminus \{x\}$ of size $r + 1$ such that the induced subgraph on $f$ is connected (which is possible since the entire set is connected and of size at least $r + 1$). Define
    $e_3 := f \cap V(\mathcal{H}).$
     Clearly, $e_3 \subseteq e_1 \cup e_2 \setminus \{x\}$. If $e_3 \in E(\mathcal{H})$, we are done. Otherwise, there exists $e_3' \in E(\mathcal{H})$ such that $e_3' \subseteq e_3 \subseteq e_1 \cup e_2 \setminus \{x\}$. This proves that $x$ is a simplicial vertex in $\mathcal{H}$.
\end{proof}

A graph is called a \emph{star-clique graph}, denoted $ \SC(x) $, if it is formed by attaching a collection of cliques $ B_1, \ldots, B_t $ to a central vertex $ x $, where each $ B_i $ is a complete graph (clique) that may vary in size.

\begin{corollary}\label{complete}
    If $ G $ is a star-clique graph, then $ \con_r(G) $ is a $ \ww $-chordal hypergraph for all $ r \geq 1 $.
\end{corollary}

\begin{proof}
Let $ V_1 $ be the set of all simplicial vertices of $ G $. Consider a minor $ \mathcal{H} $ of the hypergraph $ \con_r(G) $.  
If $ V_1 \cap V(\mathcal{H}) \neq \emptyset $, then by Lemma~\ref{first-sim}, $ \mathcal{H} $ contains a simplicial vertex.  
If $ V_1 \cap V(\mathcal{H}) = \emptyset $, then all simplicial vertices of $ G $ are excluded from $ \mathcal{H} $, which can only happen if $ \mathcal{H} $ has exactly one vertex. In that case, $ \mathcal{H} $ trivially contains a simplicial vertex.  
Therefore, every minor $ \mathcal{H} $ of $ \con_r(G) $ has a simplicial vertex, which implies that $ \con_r(G) $ is $ \ww $-chordal.
\end{proof}

The following observation highlights a key property of chordal graphs, leveraging the structure of simplicial vertices and their iterative removal.

\begin{obs}\label{chordal-nota}
Let $G$ be a chordal graph with $|V(G)| \geq 2$. A fundamental result of Dirac \cite{Dirac61} states that every chordal graph contains at least one simplicial vertex, and any chordal graph that is not complete contains at least two non-adjacent simplicial vertices.

Define $V_1$ as the set of all simplicial vertices of $G$. By Dirac's theorem, $|V_1| \geq 2$ unless $G$ is complete (in which case $|V_1| = |V(G)| \geq 2$). For $i \geq 2$, recursively define $V_i$ as the set of simplicial vertices in the subgraph $G \setminus \bigcup_{j=1}^{i-1} V_j$. Since chordality is preserved under induced subgraphs, each $G \setminus \bigcup_{j=1}^{i-1} V_j$ remains chordal, and thus each $V_i$ is non-empty. This process terminates at some step $\kappa \geq 1$ when $G \setminus \bigcup_{j=1}^{\kappa} V_j$ becomes empty.
Moreover, for all $1 \leq i < \kappa$, we have $|V_i| \geq 2$ (unless the remaining subgraph is complete, in which case all remaining vertices are simplicial). The sets ${V_i}_{i=1}^\kappa$ form a partition of $V(G)$ into pairwise disjoint subsets.
\end{obs}

The following observation explores the structure of block graphs, a subclass of chordal graphs. Using the framework established in Observation \ref{chordal-nota}, we analyze the cliques associated with vertices in $V_2$ and derive a key property regarding their structure and size.

\begin{obs}\label{block-nota}
Let $G$ be a block graph. Since every block graph is chordal, we apply the framework introduced in Observation \ref{chordal-nota}. For each vertex $v \in V_2$, there exists a maximum-sized clique in $G$ that contains $v$, such that all other vertices in this clique belong to $V_1$. Denote these cliques by $C_1, C_2, \ldots, C_m$, where $m \geq 1$.
Define 
$S_v = \bigcup_{i=1}^m \left( C_i \setminus \{v\} \right).$
As a consequence, we observe that $|S_v| \geq 1$.
\end{obs}

We now define an $n$-\emph{clique path graph}, denoted $ \cp(B_1, \ldots, B_n) $, as a graph consisting of a sequence of $n$ cliques $B_1,\ldots, B_n$, arranged linearly, with $n \geq 1$. Each consecutive pair of cliques $B_i$ and $B_{i+1}$ shares exactly one vertex, denoted $x_i$. The vertices $x_1, \ldots, x_{n-1}$ are referred to as the \emph{connecting vertices} of the $n$-clique path graph.

\begin{minipage}{\linewidth}
  %\centering
\begin{minipage}{0.33\linewidth}
\begin{figure}[H]
\begin{tikzpicture}[scale=0.45]
%\clip(-0.48,-1.07) rectangle (21.32,12.67);
\draw [line width=1pt] (2,6)-- (1,5);
\draw [line width=1pt] (1,5)-- (2,4);
\draw [line width=1pt] (2,6)-- (3,5);
\draw [line width=1pt] (3,5)-- (2,4);
\draw [line width=1pt] (1,5)-- (3,5);
\draw [line width=1pt] (2,6)-- (2,4);
\draw [line width=1pt] (3,5)-- (4,4);
\draw [line width=1pt] (5,5)-- (4,4);
\draw [line width=1pt] (3,5)-- (5,5);
\draw [line width=1pt] (5,5)-- (7,5);
\draw [line width=1pt] (8,6)-- (7,5);
\draw [line width=1pt] (8,6)-- (9,5);
\draw [line width=1pt] (7,5)-- (7,4);
\draw [line width=1pt] (7,4)-- (8,3);
\draw [line width=1pt] (9,5)-- (9,4);
\draw [line width=1pt] (9,4)-- (8,3);
\draw [line width=1pt] (8,6)-- (8,3);
\draw [line width=1pt] (7,4)-- (9,4);
\draw [line width=1pt] (7,4)-- (9,5);
\draw [line width=1pt] (7,5)-- (9,5);
\draw [line width=1pt] (7,5)-- (9,4);
\draw [line width=1pt] (9,5)-- (8,3);
\draw [line width=1pt] (7,5)-- (8,3);
\draw [line width=1pt] (7,4)-- (8,6);
\draw [line width=1pt] (9,4)-- (8,6);
\draw [line width=1pt] (10,6)-- (9,5);
\draw [line width=1pt] (9,5)-- (11,5);
\draw [line width=1pt] (10,6)-- (11,5);
\draw (1.66,7.09) node[anchor=north west] {$B_1$};
\draw (3.78,6.51) node[anchor=north west] {$B_2$};
\draw (5.52,6.25) node[anchor=north west] {$B_3$};
\draw (7.68,7.15) node[anchor=north west] {$B_4$};
\draw (9.68,7.11) node[anchor=north west] {$B_5$};
%\draw (3.32,3.01) node[anchor=north west] {$\operatorname{CP}(B_1,B_2,B_3,B_4,B_5)$};
\begin{scriptsize}
\draw [fill=black] (2,6) circle (2.5pt);
\draw [fill=black] (1,5) circle (2.5pt);
\draw [fill=black] (2,4) circle (2.5pt);
\draw [fill=black] (3,5) circle (2.5pt);
\draw[color=black] (3.02,4.54) node {$x_1$};
\draw [fill=black] (4,4) circle (2.5pt);
\draw [fill=black] (5,5) circle (2.5pt);
\draw[color=black] (5.1,4.58) node {$x_2$};
\draw [fill=black] (7,5) circle (2.5pt);
\draw[color=black] (6.52,4.7) node {$x_3$};
\draw [fill=black] (8,6) circle (2.5pt);
\draw [fill=black] (9,5) circle (2.5pt);
\draw[color=black] (9.44,4.68) node {$x_4$};
\draw [fill=black] (7,4) circle (2.5pt);
\draw [fill=black] (8,3) circle (2.5pt);
\draw [fill=black] (9,4) circle (2.5pt);
\draw [fill=black] (10,6) circle (2.5pt);
\draw [fill=black] (11,5) circle (2.5pt);
\end{scriptsize}
\end{tikzpicture}
\caption*{$\cp(B_1,B_2,B_3,B_4,B_5)$}
\end{figure}
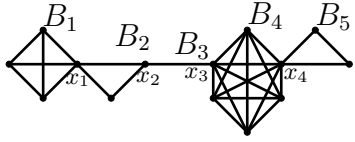
\end{minipage}
\begin{minipage}{0.64\linewidth}
For example, the graph shown left is a 5-clique path graph \( \operatorname{CP}(B_1, B_2, B_3, B_4, B_5) \), where \( B_1 \) is the clique on 4 vertices, \( B_2 \) is the clique on 3 vertices, \( B_3 \) is the clique on 2 vertices, \( B_4 \) is the clique on 6 vertices, and \( B_5 \) is the clique on 3 vertices. The vertices \( x_1, x_2, x_3, x_4 \) are the connecting vertices, which link each consecutive pair of cliques in this path.
\end{minipage}
\end{minipage}
\vskip 2mm
The following setup introduces key notation and assumptions that will be used in subsequent results. Using Observation~\ref{chordal-nota}, we define a parameter $\ell$ that plays a crucial role in the analysis.

\begin{setup}\label{allsetup}
    Let $G$ be a connected block graph. Then there exists a maximum-sized $n$-clique path, denoted $\cp(B_1, \ldots, B_n)$, as an induced subgraph with connecting vertices $x_1, x_2, \ldots, x_{n-1}$, where $n \geq 1$. Let $\hh$ be a $c'$-minor of $\con_r(G)$ for some $r \geq 1$, and assume that $|E(\hh)| \geq 3$. Using Observation~\ref{chordal-nota} and its associated notation, we define the parameter $\ell$ as:
    $\ell = \min \left\{ i \mid V_i \cap V(\hh) \neq \emptyset \right\}.$
\end{setup}

\begin{lemma}\label{not-edge}
    Assume the notation from Set-Up \ref{allsetup}, and suppose $\ell \neq 1$ and $r < n$. If $x_0, \ldots, x_{r-2} \notin V(\mathcal{H})$, where $x_0 \in V(B_1) \cap V_1$, then $x_{r-1} \in V(\mathcal{H})$. Furthermore, if $y \in N_G(x_{r-1}) \setminus \{x_{r-2}\}$, then $y \in V(\mathcal{H})$.
\end{lemma}

\begin{proof}
Suppose, for contradiction, that $x_{r-1} \notin V(\mathcal{H})$. Choose a vertex $c \in V(\mathcal{H})$ such that  
$\mathrm{d}(c, x_{r-1}) = \min\{\mathrm{d}(z, x_{r-1}) \mid z \in V(\mathcal{H})\}.$
 Set $\mathrm{d}(c, x_{r-1}) = t$. Then, there exists a shortest path  
$(x_{r-1} = y_0) y_1 \cdots y_{t-1}(y_t = c)$  
from $x_{r-1}$ to $c$. Since $c$ is the closest vertex in $V(\mathcal{H})$, all intermediate vertices $y_1, \ldots, y_{t-1}$ must lie outside $V(\mathcal{H})$.
Define  
$S = \{x_0, \ldots, x_{r-1}, y_1, \ldots, y_t\}.$  
Clearly, $|S| \geq r+1$ and $S$ is connected. Moreover, $S \cap V(\mathcal{H}) = \{c\}$. Since $c \in V(\mathcal{H})$, any $f \in E(\con_r(G))$ such that $f \subseteq S$ and $c \in f$ would imply $f \cap V(\mathcal{H}) = \{c\} \in E(\mathcal{H})$, contradicting the assumption that $\mathcal{H}$ is a $c'$-minor (i.e., no edge of $\con_r(G)$ intersects $V(\mathcal{H})$ in exactly one vertex). Hence, we must have $x_{r-1} \in V(\mathcal{H})$.

Now suppose $y \in N_G(x_{r-1}) \setminus \{x_{r-2}\}$, and assume, for contradiction, that $y \notin V(\mathcal{H})$. Then the set  
$T = \{x_0, \ldots, x_{r-1}, y\}$ 
is connected. Since $x_{r-1} \in V(\mathcal{H})$ and all other elements of $T$ lie outside $V(\mathcal{H})$, we have  
$T \cap V(\mathcal{H}) = \{x_{r-1}\}.$  
Thus, $T$ defines an edge in $\con_r(G)$ whose intersection with $V(\mathcal{H})$ is exactly $\{x_{r-1}\} \in E(\mathcal{H})$, again contradicting the $c'$-minor condition. Therefore, $y \in V(\mathcal{H})$.
\end{proof}

\begin{lemma}\label{tech-rs}
Assume the setup in \ref{allsetup} and the notation introduced in Observation \ref{block-nota}. Let $\ell = 2$, and suppose that $v_k \in V_2 \cap V(\hh)$ and $v_k \in e_1 \cap e_2$, where $e_1, e_2 \in E(\hh)$ and $e_1 \neq e_2$. Then, there exist edges $f_1, f_2 \in E(\con_r(G))$ such that the following conditions hold:
\[
S_{v_k} \subseteq f_i, \quad f_i \cap V(\hh) = e_i \quad \text{for } i = 1, 2,
\]
and
$f_1 \cup f_2 \setminus \left(S_{v_k} \cup \{v_k\}\right)$
is connected.
\end{lemma}

\begin{proof}
Let $S_{v_k} = S$. We first claim that there exists $f_1 \in E(\con_r(G))$ such that
$S \subseteq f_1$ and  $f_1 \cap V(\hh) = e_1$.
Since $\hh$ is a $c'$-minor of $\con_r(G)$, there exists an edge $f' \in E(\con_r(G))$ such that $f' \cap V(\hh) = e_1$. If $S \subseteq f'$, the claim is true. Otherwise, since every vertex of $S$ is connected to $v_k$, and $v_k \in f'$, the set $S \cup f'$ is connected, and we have $|S \cup f'| \geq r+1$. We can then choose a connected subset $f_1$ of $S \cup f'$ such that $S \subseteq f_1$ and $|f_1| = r+1$. Consequently, we have:
$f_1 \cap V(\hh) \subseteq f' \cap V(\hh) = e_1.$
If $f_1 \cap V(\hh) \subset e_1$, this would contradict the assumption that $e_1 \in E(\hh)$, since $f_1 \cap V(\hh)$ would then be a proper subset of $e_1$. Therefore, it must be that $f_1 \cap V(\hh) = e_1$.
Similarly, there exists $f_2 \in E(\con_r(G))$ such that
$S \subseteq f_2$ and  $f_2 \cap V(\hh) = e_2.$

Next, we claim that there exists $z \in f_1 \setminus (S \cup \{v_k\})$ such that $z \notin V_1$ and $z \in N_G(v_k)$. There must exist a vertex $t \in N_G(v_k)$ such that $t \notin V_1$, because if every vertex in $N_G(v_k)$ were in $V_1$, then $v_k$ would be in $V_1$, contradicting the assumption that $v_k \in V_2$. Since $e_1 \neq e_2$, there exists $x_1' \in e_1 \setminus e_2$, and hence $x_1' \in f_1$. Consider the shortest path $(v_k=a_0) a_1 \cdots a_{n-1}(a_n=x_1')$ in $G$, where $a_0, \dots, a_n \in f_1$. Then, $a_1 \in N_G(v_k)$, and if $a_1 \in V_1$, $a_1$ would be a simplicial vertex, implying the path $v_k a_2 \cdots a_{n-1} (a_n=x_1')$ would be a shorter path, leading to a contradiction. Therefore, $a_1 \notin V_1$, and hence $a_1 \notin S \cup \{v_k\}$. This proves the claim.
Similarly, there exists $z \in f_2 \setminus (S \cup \{v_k\})$ such that $z \notin V_1$ and $z \in N_G(v_k)$.

Finally, we claim that $f_1 \cup f_2 \setminus (S \cup \{v_k\})$ is connected. By the previous argument, there exist vertices $z_1 \in f_1 \setminus (S \cup \{v_k\})$ and $z_2 \in f_2 \setminus (S \cup \{v_k\})$, both of which are not in $V_1$ and are in $N_G(v_k)$. If $z_1 = z_2$, the claim follows immediately. Otherwise, if $z_1 \neq z_2$, there must be an edge between $z_1$ and $z_2$ because $v_k \in V_2$, which ensures that $z_1$ and $z_2$ are connected. Therefore, $f_1 \cup f_2 \setminus (S \cup \{v_k\})$ is connected, and the proof is complete.
\end{proof}

\begin{notation}\label{nota-tech-rs1}
   Assume the setup in \ref{allsetup} and the notation introduced in Observation~\ref{block-nota}. Let $\ell = 2$, and suppose that $v_k \in V_2 \cap V(\hh)$ and $v_k \in e_1 \cap e_2$, where $e_1, e_2 \in E(\mathcal{H})$ and $e_1 \neq e_2$. By Lemma~\ref{tech-rs}, there exist edges $f_1, f_2 \in E(\con_r(G))$ such that:
   \[
      S_{v_k} \subseteq f_i, \quad f_i \cap V(\mathcal{H}) = e_i, \quad \text{for } i = 1, 2,
   \]
   and the set
   $f_1 \cup f_2 \setminus (S_{v_k} \cup \{v_k\})$
   induces a connected subgraph.
   Next, fix an element $x_1' \in e_1 \setminus e_2$. For each \(1 \leq i \leq \gamma\), where \(\gamma \geq 1\), define a set
   $C_i = \{c_{i,1}, \ldots, c_{i,q_i}\} \subseteq V(G),$
   satisfying the following conditions:
\begin{equation}\label{cond}
      \left.
      \begin{array}{ll}
      \text{(i)} & C_i \text{ is connected}, \\
      \text{(ii)} & c_{i,j} \in N_G(x_1') \text{ for some } j \in \{1, \ldots, q_i\}, \\
      \text{(iii)} & C_i \cap V(\mathcal{H}) = \emptyset, \\
      \text{(iv)} & \dd(c_{i,j}, v_k) > \dd(v_k, x_1') \text{ for all } j \in \{1, \ldots, q_i\}.
      \end{array}
      \right\}
   \end{equation}
   
   Finally, define
   $C = \bigcup_{i=1}^{\gamma} C_i.$
\end{notation}

\begin{lemma}\label{tech-rs1}
 Let the notation be as in Notation~\ref{nota-tech-rs1}. Then $v_k$ is a simplicial vertex of $\mathcal{H}$ in either of the following cases: if $|f_1 \cup f_2 \setminus (S_{v_k} \cup \{v_k\})| \geq r+1$, or if $|f_1 \cup f_2 \setminus (S_{v_k} \cup \{v_k\})| \leq r$ and $|C| \geq |S_{v_k}|$.

   % Using the notation from Notation \ref{nota-tech-rs1}, we have:
   % \begin{enumerate}
   %     \item If $|f_1 \cup f_2 \setminus (S_{v_k} \cup \{v_k\})| \geq r+1$, then $v_k$ is a simplicial vertex of $\mathcal{H}$.
   %     \item If $|f_1 \cup f_2 \setminus (S_{v_k} \cup \{v_k\})| \leq r$ and $|C| \geq |S_{v_k}|$, then $v_k$ is a simplicial vertex of $\mathcal{H}$.
   % \end{enumerate}
\end{lemma}

\begin{proof}
(1) Suppose that $\left| (f_1 \cup f_2) \setminus (S_{v_k} \cup \{v_k\}) \right| \geq r+1$. Then, there exists a connected subset $f \subseteq (f_1 \cup f_2) \setminus (S_{v_k} \cup \{v_k\})$ with $|f| = r+1$. Let $e_3 = f \cap V(\mathcal{H})$; since $f$ avoids $\{v_k\}$, we have $e_3 \subseteq (e_1 \cup e_2) \setminus \{v_k\}$. If $e_3$ is itself an edge in $E(\mathcal{H})$, then $v_k$ must be a simplicial vertex. Otherwise, by the properties of the hypergraph $\mathcal{H}$, there exists a subset $e_3' \in E(\mathcal{H})$ such that $e_3' \subseteq e_3$, which again implies that $v_k$ is simplicial. In either case, the conclusion follows.

\vskip 1mm
\noindent
(2) Now assume $|A| \leq r$, where $A = f_1 \cup f_2 \setminus (S_{v_k} \cup \{v_k\})$. Let $S = S_{v_k}$ and define $D' = f_2 \setminus (S \cup \{v_k\})$. We claim that $D'$ is a connected set.
Let $\alpha, \beta \in D'$. If $\alpha$ and $\beta$ are not adjacent, since $\alpha, \beta \in f_2$, there exists a shortest path  
$\alpha \ y_1 \cdots y_p \ \beta$
within $f_2$, where each $y_i \in f_2$ for $1 \leq i \leq p$. Since $y_i$ is not a simplicial vertex of $G$ for all $1 \leq i \leq p$, we have $y_i \notin S \cup \{v_k\}$, ensuring that $D'$ remains connected.

Next, we claim that $D' \cap C = \emptyset$. Suppose, for contradiction, that there exists $c_{i,j} \in D' \cap C$, where $c_{i,j} \in C_i$ for some $1 \leq i \leq \gamma$, $1 \leq j \leq q_i$.
Since $v_k, x_1' \in f_1$, there exists a shortest path  
$(v_k = z_0) z_1 \cdots z_{t-1} (z_t = x_1')$
in $G$, where $z_i \in f_1$ for all $0 \leq i \leq t$. Let $B$ be the maximal clique containing $z_{t-1}$ and $x_1'$. Similarly, since $v_k, c_{i,j} \in f_2$, there exists a shortest path $P(v_k, c_{i,j})$ within $f_2$.
Since $c_{i,j} \in C_i$ and conditions (i) and (ii) from \eqref{cond} hold, there exists a shortest path  
$x_1' a_1 \cdots a_m c_{i,j}$,
where $a_p \in C_i$ for all $1 \leq p \leq m$. We observe that $a_1 \notin V(B)$, since otherwise, $\dd(v_k, a_1) \leq \dd(v_k, x_1')$, contradicting condition (iv) from \eqref{cond}.
Since $x_1'$ is a cut vertex, $a_1$ and $z_{t-1}$ belong to different components of $G \setminus \{x_1'\}$. However, the path  
$z_{t-1} \cdots v_k \ P(v_k, c_{i,j}) \ a_m \cdots a_1$
remains in $G \setminus \{x_1'\}$, which is a contradiction. Therefore, $D' \cap C = \emptyset$.

Since $D' \cap \{x_1'\} = \emptyset$ (because $x_1' \in e_2$), we have  
$|f_2| = r + 1 = |D'| + |S| + 1.$
Now, consider $D' \cup R \cup C$, where $R = \{z_1, \ldots, z_{t-1}, z_t\}$. We can show that $D' \cup R \cup C$ is connected. By computing its size, we obtain  
$|D' \cup R \cup C| \geq r + 1.$ 
Select a connected subset $f \subseteq D' \cup R \cup C$ with $|f| = r + 1$, and define  
$e = f \cap V(\mathcal{H}).$
This ensures that $e \subseteq e_1 \cup e_2 \setminus \{v_k\}$. If $e \in E(\mathcal{H})$, then $v_k$ is a simplicial vertex. Otherwise, there exists $e_s \in E(\mathcal{H})$ such that $e_s \subseteq e$, proving that $v_k$ is a simplicial vertex.
\end{proof}

\section{Shellability of Higher Independence Complexes} \label{con-chor}
This section establishes sufficient conditions on $r$ and on the graph $G$ under which the $r$-independence complex $\mathcal{I}_r(G)$ is shellable.
We begin by examining a claim from~\cite{FPSAA23}, where it was asserted that for any tree $G$, the hypergraph $\mathcal{C}_r(G)$ is $W$-chordal for all $r \geq 2$ (see~\cite[Theorem~2.7]{FPSAA23}). However, the proof presented there contains a flaw. In fact, we exhibit an explicit family of trees for which $\mathcal{C}_r(G)$ is not $W$-chordal. Moreover, for each $r \geq 4$, we construct a family of graphs $G_t$ such that the hypergraph $\mathcal{C}_r(G_t)$ fails to be $W$-chordal. The construction is given below.

\begin{cons}
    Let $H_1$, $H_2$, and $H_3$ be connected graphs with pairwise disjoint vertex sets, each satisfying $|V(H_i)| \geq t$ for $1 \leq i \leq 3$. Define a base graph $G$ with vertex set  
$V(G) = \{v_1, v_2, v_3, v_4, v_5, v_6, v_7\}$
and edge set  
\[
E(G) = \{\{v_1, v_2\}, \{v_2, v_3\}, \{v_1, v_4\}, \{v_4, v_5\}, \{v_1, v_6\}, \{v_6, v_7\}\}.
\]
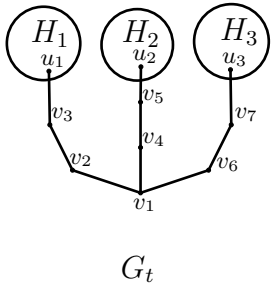
 \begin{minipage}{\linewidth}
  %\centering
\begin{minipage}{0.3\linewidth}

 \begin{figure}[H]
\begin{tikzpicture}[scale=0.3]
%\clip(3.6860088691796156,-1.6862749445676257) rectangle (39.250886917960266,24.220066518846984);
\draw [line width=1pt] (16,5)-- (13,6);
\draw [line width=1pt] (16,5)-- (16,7);
\draw [line width=1pt] (16,5)-- (19,6);
\draw [line width=1pt] (13,6)-- (12,8);
\draw [line width=1pt] (16,7)-- (16,9);
\draw [line width=1pt] (19,6)-- (20,8);
\draw [line width=1pt] (15.854390243902438,11.526585365853657) circle (1.5754364474646385cm);
\draw [line width=1pt] (20.008935698447985,11.67470066518846) circle (1.6465297539904094cm);
\draw [line width=1pt] (11.724059268292685,11.588849512195122) circle (1.6188855999151566cm);
\draw [line width=1pt] (12,8)-- (11.96015521064307,10.365432372505534);
\draw [line width=1pt] (16,9)-- (16.016740576496748,10.558603104212851);
\draw [line width=1pt] (20,8)-- (19.966008869179692,10.429822616407973);
\draw (10.66991130820404,13.11470066518846) node[anchor=north west] {$H_1$};
\draw (14.676740576496741,13.107871396895777) node[anchor=north west] {$H_2$};
\draw (19.019179600887003,13.179090909090899) node[anchor=north west] {$H_3$};
\begin{scriptsize}
\draw [fill=black] (16,5) circle (2.5pt);
\draw[color=black] (16.17771618625284,4.471773835920173) node {$v_1$};
\draw [fill=black] (13,6) circle (2.5pt);
\draw[color=black] (13.37283813747234,6.459090909090904) node {$v_2$};
\draw [fill=black] (16,7) circle (2.5pt);
\draw[color=black] (16.57771618625284,7.367871396895781) node {$v_4$};
\draw [fill=black] (19,6) circle (2.5pt);
\draw[color=black] (19.861130820399197,6.259090909090904) node {$v_6$};
\draw [fill=black] (12,8) circle (2.5pt);
\draw[color=black] (12.564057649667457,8.455188470066512) node {$v_3$};
\draw [fill=black] (16,9) circle (2.5pt);
\draw[color=black] (16.57771618625284,9.263968957871389) node {$v_5$};
\draw [fill=black] (20,8) circle (2.5pt);
\draw[color=black] (20.669911308204078,8.455188470066512) node {$v_7$};
\draw [fill=black] (11.96015521064307,10.365432372505534) circle (2.5pt);
\draw[color=black] (12.142594235033311,10.816164079822606) node {$u_1$};
\draw [fill=black] (16.016740576496748,10.558603104212851) circle (2.5pt);
\draw[color=black] (16.19917960088699,11.009334811529923) node {$u_2$};
\draw [fill=black] (19.966008869179692,10.429822616407973) circle (2.5pt);
\draw[color=black] (20.148447893569934,10.880554323725045) node {$u_3$};
\end{scriptsize}
\end{tikzpicture}
\caption*{$G_t$}
\end{figure}
\end{minipage}
\begin{minipage}{0.65\linewidth}
The graph $G_t$ is formed by taking the disjoint union of $G$, $H_1$, $H_2$, and $H_3$, and then attaching each $H_i$ to $G$ via a single edge. Specifically, let $u_1 \in V(H_1)$, $u_2 \in V(H_2)$, and $u_3 \in V(H_3)$ be arbitrary vertices. Then the vertex and edge sets of $G_t$ are given by  
$V(G_t) = V(G) \cup V(H_1) \cup V(H_2) \cup V(H_3),$  
$E(G_t) = E(G) \cup E(H_1) \cup E(H_2) \cup E(H_3) \cup \{\{v_3, u_1\}, \{v_5, u_2\}, \{v_7, u_3\}\}.$
\end{minipage}
\end{minipage}
\end{cons}

\begin{proposition}\label{not-chordal}
If $t \geq r - 3$ for some $r \geq 4$, then the hypergraph $\con_r(G_t)$ is not $\ww$-chordal.
\end{proposition}

\begin{proof}
For each $1 \leq i \leq 3$, choose a subset $S_i \subseteq V(H_i)$ with $|S_i| = r - 3$ such that the following subsets of $V(G_t)$ belong to $E(\con_r(G_t))$: 
$S_1 \cup \{v_1, v_2, v_3, v_4\}$, 
$S_1 \cup \{v_1, v_2, v_3, v_6\}$, 
$S_2 \cup \{v_5, v_4, v_1, v_6\}$, 
$S_2 \cup \{v_5, v_4, v_1, v_2\}$, 
$S_3 \cup \{v_7, v_6, v_1, v_4\}$, 
and $S_3 \cup \{v_7, v_6, v_1, v_2\}$.  

Let $S'_i = V(H_i) \setminus S_i$, and define $S = S'_1 \cup S'_2 \cup S'_3$. Let $\mathcal{H}$ be the hypergraph obtained from $\con_r(G_t)$ by deleting all vertices in $S$ and contracting the vertices in $S_1 \cup S_2 \cup S_3$. Since vertex deletion and contraction behave compatibly with $\con_r(G)$ (that is, $\con_r(G) \setminus x = \con_r(G \setminus x)$ for any vertex $x \in V(G)$), the resulting hypergraph $\mathcal{H}$ has vertex set 
$V(\mathcal{H}) = \{v_1, v_2, v_3, v_4, v_5, v_6, v_7\}$
and edge set 
$E(\mathcal{H}) =$ $ \{\{v_1, v_2, v_3, v_4\}$, $\{v_1, v_2, v_3, v_6\}$, $\{v_5, v_4, v_1, v_6\}$, $\{v_5, v_4, v_1, v_2\}$, $\{v_7, v_6, v_1, v_4\}$, $\{v_7, v_6, v_1, v_2\}\}$.
One can verify that $\mathcal{H}$ contains no simplicial vertices. Therefore, $\mathcal{H}$ is not $\ww$-chordal, and consequently, $\con_r(G_t)$ is not $\ww$-chordal.
\end{proof}

\begin{remark}\label{counter}
In~\cite{FPSAA23}, it was claimed that for any tree $G$, the hypergraph $\con_r(G)$ is $\ww$-chordal for all $r \geq 2$ (see~\cite[Theorem~2.7]{FPSAA23}), and consequently that the $r$-independence complex $\ind_r(G)$ is shellable. However, the argument provided there appears to be incomplete. Proposition~\ref{not-chordal} shows that for every $r \geq 4$, there exists a tree $G$ (constructed by taking each $H_i$ to be a tree for $i=1,2,3$) such that $\con_r(G)$ is \emph{not} $\ww$-chordal. This provides a counterexample to the claim and calls into question the corresponding shellability conclusion.
\end{remark}

The following theorem provides a sufficient condition for the $r$-independence complex $\ind_r(G)$ to be shellable when $r$ is close to the order of the graph.

\begin{theorem}\label{chordal-cond}
Let $G=(V(G),E(G))$ be a graph on $n=|V(G)| \geq 3$ vertices. If $n-2 \leq r \leq n-1$, then $\ind_r(G)$ is shellable.
\end{theorem}

\begin{proof}
By Theorem \ref{russ-result}, it is enough to prove that $\con_r(G)$ is $\ww$-chordal. If $r = n - 1$, then each edge of $\con_r(G)$ has size $n$, which implies that there can be at most one such edge. Hence, $\con_r(G)$ is trivially $\ww$-chordal.
Now suppose $r = n - 2$. Then each edge in $\con_r(G)$ has size $n - 1$. If $|E(\con_r(G))| \leq 2$, then again the hypergraph is trivially $\ww$-chordal. Assume $|E(\con_r(G))| \geq 3$. We first show that for any two distinct edges $f_1, f_2 \in \con_r(G)$, we have $|f_1 \cap f_2| = n - 2$.
Since $|f_1| = |f_2| = n - 1$, we have
$|f_1 \cup f_2| = 2(n - 1) - |f_1 \cap f_2|.$
But \(f_1 \cup f_2 \subseteq V(G)\), so 
$2(n - 1) - |f_1 \cap f_2| \leq n \Rightarrow |f_1 \cap f_2| \geq n - 2.$
Also, since \(f_1 \neq f_2\), we cannot have \(|f_1 \cap f_2| = n - 1\). Hence, \(|f_1 \cap f_2| = n - 2\), and so \(f_1 \cup f_2 = V(G)\).

Now let $\hh$ be any minor of $\con_r(G)$. We show that $\hh$ contains a simplicial vertex. If $|E(\hh)| \leq 2$, the result is immediate. Assume $|E(\hh)| \geq 3$. Let $e_1, e_2 \in E(\hh)$ with $e_1 \neq e_2$, and choose vertices $x_1' \in e_1 \setminus e_2$ and $x_2' \in e_2 \setminus e_1$.
We claim that $x_1'$ is a simplicial vertex in $\hh$. Let $e_1'$ and  $e_2'$ be two distinct edges in $ E(\mathcal{H})$ such that their intersection $e_1' \cap e_2'$ contains the vertex $x_1'$. Since $x_1' \in e_1 \setminus e_2$, it follows that $e_1', e_2' \neq e_2$. Let $f_1', f_2', g_2 \in E(\con_r(G))$ be preimages of $e_1', e_2', e_2$ under the minor operation, so that $f_i' \cap V(\hh) = e_i'$ for $i = 1, 2$ and $g_2 \cap V(\hh) = e_2$.
From earlier, we know $f_1' \cup f_2' = V(G)$, and  we have
$e_2 = g_2 \cap V(\hh) \subseteq (f_1' \cup f_2') \cap V(\hh) = e_1' \cup e_2' \subseteq V(\hh).$
Since $x_1' \notin e_2$, this implies $e_2 \subseteq (e_1' \cup e_2') \setminus \{x_1'\}$. Hence, $x_1'$ is a simplicial vertex of $\hh$.
Therefore, every minor of $\con_r(G)$ contains a simplicial vertex, and so $\con_r(G)$ is $\ww$-chordal.
\end{proof}

In the next lemma, we establish a connection between a cut vertex of the tree $G$ and a $c'$-minor of the graph $\con_r(G)$, where $r \geq 1$.  

\begin{lemma}\label{lower-lemma}
Let $G$ be a tree, and let $\mathcal{H}$ be a $c'$-minor of $\con_r(G)$ for some $r \geq 1$. Following the notation in Set-Up~\ref{allsetup}, assume that $\ell \neq 1$. Then there exists a cut vertex $v$ of $G$ such that $v \in V(\mathcal{H})$ and $V(\mathcal{H}) \setminus \{v\} \subseteq V(G_1)$, where $G_1$ is a connected component of $G \setminus \{v\}$.
% Let $G$ be a tree, and let $\mathcal{H}$ be a $c'$-minor of $\con_r(G)$ for some $r \geq 1$.  
% Following the notation in Set-Up~\ref{allsetup}, assume $\ell \neq 1$.  
% Then there exists a cut vertex $v$ of $G$ such that:
% \begin{enumerate}
%     \item $v \in V(\mathcal{H})$, and
%     \item $V(\mathcal{H}) \setminus \{v\} \subseteq V(G_1)$,
% \end{enumerate}
% where $G_1$ is a connected component of $G \setminus \{v\}$.
\end{lemma}

\begin{proof}
Let $v \in V_\ell \cap V(\mathcal{H})$. Since $\mathcal{H}$ is a $c'$-minor of $\con_r(G)$, the vertex $v$ cannot appear as a singleton edge in $\mathcal{H}$, i.e., $\{v\} \notin E(\mathcal{H})$. Hence, $v$ must be adjacent to some vertex in $G \setminus \bigcup_{i=1}^{\ell-1} V_i$.  
As $v$ is simplicial in the induced subgraph $G \setminus \bigcup_{i=1}^{\ell-1} V_i$, its neighborhood in this subgraph is a singleton:
$N_{G \setminus \bigcup_{i=1}^{\ell-1} V_i}(v) = \{a\},$
for some $a \in V(G)$. Since $G$ is a tree and $v$ is a cut vertex, the deletion of $v$ produces at least two connected components. Let $G_1$ be the component containing $a$.
Now consider any vertex $x \in V(\mathcal{H}) \setminus \{v\}$. Because $G$ is a tree, there exists a unique path from $v$ to $x$, say
$v z_1 z_2 \cdots z_t x.$
Since $v \in V_\ell$, there exists a path $a_1 a_2 \cdots a_{\ell-1} v$ with $a_i \in V_i$ for $1 \leq i \leq \ell - 1$. Moreover, as $x \in V(\mathcal{H}) \setminus \{v\}$ and $x \in V_j$ for some $j \geq \ell$, there exists a path $x b_1 b_2 \cdots b_r$ with $r \geq \ell - 1$.  

Let $z_i$ be an internal vertex on the path from $v$ to $x$. Then both  
$a_1 a_2 \cdots a_{\ell-1} v z_1 \cdots z_{i-1} z_i$ and  $z_i z_{i+1} \cdots z_t x b_1 \cdots b_r$
are paths of length at least $\ell$. Consequently, $z_i \notin V_s$ for $1 \leq s \leq \ell$, which means all internal vertices lie outside $\bigcup_{s=1}^{\ell} V_s$.
Since $a$ is the unique neighbor of $v$ in $G \setminus \bigcup_{s=1}^{\ell-1} V_s$, the first edge on the path from $v$ to $x$ must be $v z_1 = va$, so $z_1 = a$. The remainder of the path connects $a$ to $x$ within $G \setminus \{v\}$, which shows that $x$ lies in the same connected component as $a$, namely $G_1$.  
As $x \in V(\mathcal{H}) \setminus \{v\}$ was arbitrary, we conclude that
$V(\mathcal{H}) \setminus \{v\} \subseteq V(G_1),$
as desired.
\end{proof}

We now extend Theorem \ref{chordal-cond} to establish a lower bound in the case when $G$ is a forest. Let $G$ be a tree, and let $\mathcal{H}$ be a $c'$-minor of $\con_r(G)$ for some $r \geq 1$. Consider the cut vertex $v$ of $G$ as given by Lemma~\ref{lower-lemma}, and let $G_1$ be the component of $G \setminus \{v\}$ that contains $V(\mathcal{H}) \setminus \{v\}$. Define $C_v = V(G) \setminus (V(G_1) \cup \{v\})$, ensuring that $C_v \cap V(\mathcal{H}) = \emptyset$.

\begin{theorem}\label{tree-lower}
Let $G$ be a forest. Then, for all $r \geq n - 5$, the complex $\ind_r(G)$ is shelable.
\end{theorem}

\begin{proof}
By Theorem \ref{russ-result}, it is enough to prove that $\con_r(G)$ is $\ww$-chordal.
Since forest graphs form a hierarchy, and by Proposition~\ref{del-ope}, it suffices to prove that every $c'$-minor of $\con_r(G)$ contains a simplicial vertex when $G$ is a connected graph. Let $\mathcal{H}$ be a $c'$-minor of $\con_r(G)$, and use the notation from Set-Up~\ref{allsetup}. If $\ell = 1$, then by Lemma~\ref{first-sim}, $\mathcal{H}$ contains a simplicial vertex. Hence, assume that $\ell \neq 1$.
Define the set of vertices
$S = \{a \in V(G) \mid \text{$a$ satisfies Lemma~\ref{lower-lemma}}\} = \{a_1, \dots, a_m\}.$
Let $a_k \in S$ be such that
$|C_{a_k}| = \min \{\, |C_{a_i}| \mid i = 1, \dots, m \, \}.$
We claim that $a_k$ is a simplicial vertex of $\mathcal{H}$.
Suppose that $a_k \in e_1 \cap e_2$ for distinct $e_1, e_2 \in E(\mathcal{H})$. By an argument similar to Lemma~\ref{tech-rs}, we can choose $f_1, f_2 \in E(\con_r(G))$ such that
$C_{a_k} \subseteq f_i \quad \text{and} \quad e_i = f_i \cap V(\mathcal{H}) \quad \text{for } i = 1, 2.$
Next, choose
$x_1 \in e_1 \setminus e_2$ and $x_2 \in e_2 \setminus e_1$
such that
$\dd(x_1, a_k) \geq \dd(x, a_k)$ for all  $x \in e_1 \setminus e_2$ and $\dd(x_2, a_k) \geq \dd(y, a_k)$ for all  $y \in e_2 \setminus e_1$.
Since $a_k \in f_1$ and $x_1 \in f_1$, there exists a path $a_k y_1 \cdots y_p x_1$ in $G$ with $y_i \in f_1$. Similarly, there exists a path $a_k z_1 \cdots z_q x_2$ with $z_j \in f_2$. 
If $y_1 \neq z_1$, then $x_1$ and $x_2$ lie in different connected components of $G \setminus a_k$. However, this contradicts the fact that $x_1$ and $x_2$ belong to the same connected component of $G \setminus a_k$.
Therefore, we must have $y_1 = z_1$, and we denote this common vertex by $a$. This implies $a \in f_1 \cap f_2$.

We now claim that either $x_1$ or $x_2$ belongs to $S$. Suppose, for contradiction, that neither $x_1$ nor $x_2$ belongs to $S$. For $i = 1, 2$, define
\[
D_i = \{\, y \in (V(G)\setminus x_i) \mid \text{the induced path } P(y, a_k) \text{ contains } x_i \text{ as an internal vertex} \, \}.
\]
Note that $D_1 \cap f_2 = \emptyset$ and $D_2 \cap f_1 = \emptyset$; otherwise, $x_1 \in e_2$ or $x_2 \in e_1$, contradicting the choice of $x_1$ and $x_2$. If $D_i \cap V(\mathcal{H}) = \emptyset$ for both $i = 1, 2$, then, since $x_1$ and $x_2$ are cut vertices satisfying Lemma~\ref{lower-lemma}, we must have $x_1, x_2 \in S$, a contradiction. Hence, $D_i \cap V(\mathcal{H}) \neq \emptyset$ for both $i$.
Let $x_i' \in D_i \cap V(\mathcal{H})$ for $i = 1, 2$. 
%Then clearly $x_1' \notin f_2$ and $x_2' \notin f_1$, so $x_1' \notin e_2$ and $x_2' \notin e_1$.
Moreover, $x_1' \notin f_1$ by the maximality condition on $x_1$.
Since $x_i' \in V(\mathcal{H})$ and $\ell \neq 1$, there exists $x_i'' \in N_G(x_i')$ such that
$\dd(x_i'', a_k) > \dd(x_i', a_k) \text{ for }i=1,2.$
Also, $x_2', x_2'' \in D_2$, and since $D_2 \cap f_1 = \emptyset$, we have $x_2', x_2'' \notin f_1$.
If $x_1'' \in f_1$, then $x_1' \in f_1$, a contradiction. Thus, $x_1'' \notin f_1$. This implies that $f_1$ misses at least five vertices: $x_1', x_1'', x_2, x_2', x_2'' \notin f_1$, contradicting the assumption $|f_1| = r + 1 \geq n - 4$. Therefore, either $x_1 \in S$ or $x_2 \in S$. Without loss of generality, assume $x_1 \in S$. 

Let $D_1 = C_{x_1}$. Define
$D = f_2 \setminus (C_{a_k} \cup \{a_k\}).$
By construction of $C_{a_k}$, we have $D \subseteq G_1$, where $G_1$ is the connected component of $G \setminus a_k$ containing $V(\mathcal{H}) \setminus \{a_k\}$.
Now, we prove that the set
$S' = D \cup \{a, y_2, \ldots, y_p, x_1\} \cup D_1$
is connected. Let $\alpha, \beta \in D$. Then there exists a path in $G_1$, hence in $G$, hence in $f_2$, and no internal vertex lies in $C_{a_k} \cup \{a_k\}$, so $D$ is connected. Let $c \in D_1$ and $d \in D$. Then there exists a path $P(d,a)$ in $f_2 \setminus (C_{a_k} \cup \{a_k\})$, and a path $a y_2 \cdots y_p x_1 r_1 \cdots r_t c$ where $r_i \in D_1$. Thus, $S'$ is connected.

Observe that $f_2 = D \cup C_{a_k} \cup \{a_k\}$, so
$r + 1 = |D| + |C_{a_k}| + 1.$
Also, $|S'| \geq r + 1$, so we can choose $f \in E(\con_r(G))$ such that $f \subseteq S'$. Set $e = f \cap V(\mathcal{H})$. Clearly, $e \subseteq e_1 \cup e_2 \setminus \{a_k\}$. If $e \in E(\mathcal{H})$, we are done. Otherwise, there exists $e_s \in E(\mathcal{H})$ such that $e_s \subseteq e_1 \cup e_2 \setminus \{a_k\}$. This proves that $a_k$ is a simplicial vertex.
\end{proof}

We now turn to the case $r=2$, where a clean structural condition on $G$ guarantees the shellability of $\ind_2(G)$.

\begin{theorem}\label{block-2}
   If $G$ is a block graph, then $\ind_2(G)$ is shellable.
\end{theorem}

\begin{proof}
By Theorem \ref{russ-result}, it is enough to prove that $\con_2(G)$ is $\ww$-chordal.
Block graphs admit a hierarchical structure. By Proposition~\ref{del-ope}, it suffices to show that every $c'$-minor of $\con_2(G)$ contains a simplicial vertex, assuming $G$ is connected.
Let $\hh$ be a $c'$-minor of $\con_2(G)$. Without loss of generality, assume $G$ is connected and adopt the notation from Set-up~\ref{allsetup}.  
If $1 \leq n \leq 2$, then $G$ is a star-clique graph, and by Corollary~\ref{complete}, $\hh$ has a simplicial vertex. Now assume $n \geq 3$.
If $\ell = 1$, then Lemma~\ref{first-sim} implies that $\hh$ has a simplicial vertex. So suppose $\ell \neq 1$. In particular, $x_0 \notin V(\hh)$, where $x_0 \in V(B_1) \cap V_1$.
By Lemma~\ref{not-edge}, $x_1 \in V(\hh)$, and since $x_2 \in N_G(x_1) \setminus \{x_0\}$, we also have $x_2 \in V(\hh)$.

We now show that $|B_1| = 2$. If $|B_1| \geq 3$, then $\{x_1\}$ would be an edge of $\hh$, contradicting the definition of a $c'$-minor. So $|B_1| = 2$.
If $V(B_2) = \{x_1, x_2\}$, then $x_1$ appears in only one hyperedge of $\hh$, hence is simplicial.  
Otherwise, suppose $V(B_2) = \{x_1, x_2, \beta_1, \dots, \beta_p\}$ for some $p \geq 1$. Then for each $i$, we have $\beta_i \in V_2 \cap V(\hh)$ and $|N_G(\beta_i) \cap V_1| = 1$. Let $N_G(\beta_i) \cap V_1 = \{z_i\}$.
Now consider $x_1 \in e_1 \cap e_2$, where $e_1 \neq e_2$ and $e_i \in E(\hh)$ for $i = 1,2$. Since $x_1 \in V_2$, Observation~\ref{block-nota} implies $S_{x_1} = \{x_0\}$.
By Lemma~\ref{tech-rs}, there exist $f_1, f_2 \in E(\con_2(G))$ such that $S_{x_1} \subseteq f_i$ and the set $f_1 \cup f_2 \setminus \{x_0, x_1\}$ is connected.
Each $f_i$ is either $\{x_0, x_1, x_2\}$ or $\{x_0, x_1, \beta_j\}$ for some $j$. Let $x_1' \in e_1 \setminus e_2$ and $x_2' \in e_2 \setminus e_1$. Without loss of generality, assume $x_1' = \beta_j$ for some $j$. Then
$|f_1 \cup f_2 \setminus \{x_0, x_1\}| = 2 < r+1 = 3.$
Define $C = \{z_j\}$. Then $C$ satisfies the conditions in \eqref{cond}. Hence, by Lemma~\ref{tech-rs1}, $x_1$ is a simplicial vertex of $\hh$.
\end{proof}

Beyond the case $r=2$, one can obtain shellability of $\ind_r(G)$ for all $r \geq 2$ under an additional structural restriction on $G$, namely a bound on its diameter.

\begin{theorem}\label{block-diam}
If $G$ is a block graph with $\diam(G) \leq 4$, then $\ind_r(G)$ is shellable for all $r \geq 2$.
\end{theorem}

\begin{proof}
By Theorem \ref{russ-result}, it is enough to prove that $\con_r(G)$ is $\ww$-chordal.
Let $\hh$ be a $c'$-minor of $\con_r(G)$ for some $r \geq 2$. We claim that $\hh$ contains a simplicial vertex. To prove this, assume that $G$ is connected and adopt the notation from Set-up~\ref{allsetup}. We consider cases based on the number of blocks $n$.

\smallskip
\noindent\textit{Case 1: $1 \leq n \leq 2$.}  
In this case, $G$ is a star-clique graph. By Corollary~\ref{complete}, $\hh$ contains a simplicial vertex.

\smallskip
\noindent\textit{Case 2: $3 \leq n \leq 4$.}  
If $r = 2$, then by Theorem~\ref{block-2}, $\con_2(G)$ is $\ww$-chordal, so $\hh$ contains a simplicial vertex.

Now assume $r \geq 3$. We analyze based on the value of $\ell$.

\smallskip
\noindent\textit{Subcase: $\ell = 1$.}  
Then, by Lemma~\ref{first-sim}, $\hh$ has a simplicial vertex.

\smallskip
\noindent\textit{Subcase: $\ell = 2$.}  
Let $V_2 \cap V(\hh) = \{v_1, \dots, v_s\}$. By Observation~\ref{block-nota}, define
$|S_{v_k}| = \min \left\{ |S_{v_i}| : 1 \leq i \leq s \right\}.$
We claim that $v_k$ is simplicial in $\hh$.
Let $e_1, e_2 \in E(\hh)$ be distinct edges such that $v_k \in e_1 \cap e_2$. By Lemma~\ref{tech-rs}, there exist $f_1, f_2 \in E(\con_r(G))$ such that
$S_{v_k} \subseteq f_i$,  $f_i \cap V(\hh) = e_i$ for  $i = 1, 2$,
and the set $f_1 \cup f_2 \setminus (S_{v_k} \cup \{v_k\})$ is connected.

We analyze this for different values of $n$:

\smallskip
\noindent\textit{Case $n = 3$.}  
Here, $V_2 = B_2 \setminus V_1$. Let $x_1' \in e_1 \setminus e_2$ and $x_2' \in e_2 \setminus e_1$. Since $x_1', x_2' \in V_2$, there are two possibilities:
If $\left| f_1 \cup f_2 \setminus (S_{v_k} \cup \{v_k\}) \right| \geq r+1$, then by Lemma~\ref{tech-rs1}, $v_k$ is simplicial.
 If $\left| f_1 \cup f_2 \setminus (S_{v_k} \cup \{v_k\}) \right| \leq r$, then $x_1' = v_j$ for some $j$, and we let $C = S_{v_j}$. Then each component of $C$ satisfies the conditions in \eqref{cond}, and again by Lemma~\ref{tech-rs1}, $v_k$ is simplicial.

\smallskip
\noindent\textit{Case $n = 4$.}  
Here, $V_2 = N_G(x_2) \setminus V_1$. Let $x_1' \in e_1 \setminus e_2$ and $x_2' \in e_2 \setminus e_1$. If $x_1', x_2' \notin V_2$, then they belong to $V_t$ for $t \geq 3$, which is impossible since $V_3 = \{x_2\}$ and $V_t = \emptyset$ for $t \geq 4$. So either $x_1'$ or $x_2'$ lies in $V_2$, and we can proceed as in the $n = 3$ case to construct an edge $e_3 \subseteq (e_1 \cup e_2) \setminus \{v_k\}$, confirming that $v_k$ is simplicial.

\smallskip
\noindent
Note that $\ell = 3$ does not occur. For $n = 3$, we have $V_3 = \emptyset$, and for $n = 4$, $V_3 = \{x_2\}$ contradicts the definition of a $c'$-minor. Also, $V_t = \emptyset$ for all $t \geq 4$.

\smallskip
Thus, every $c'$-minor of $\con_r(G)$ contains a simplicial vertex. Since every induced subgraph of $G$ is a block graph with diameter at most $4$, Proposition~\ref{del-ope} implies that $\con_r(G)$ is a $\ww$-chordal hypergraph for all $r \geq 2$.
\end{proof}

Let $G$ be a tree, and let $P := y_1 y_2 \cdots y_n$ be a longest path in $G$, called the \emph{central path}, where $y_1$ and $y_n$ are leaves. The tree $G$ is a \emph{caterpillar} if every vertex $z \notin V(P)$ is adjacent to an internal vertex of $P$, i.e., $\mathrm{d}(z, y_i) = 1$ for some $i \in \{2,\ldots,n-1\}$. More generally, $G$ is a \emph{lobster} if every vertex $z \notin V(P)$ either satisfies $\mathrm{d}(z, y_i) = 1$ for some $i \in \{2,\ldots,n-1\}$ or $\mathrm{d}(z, y_i) = 2$ for some $i \in \{3,\ldots,n-2\}$.

The following result shows that, more generally, every forest of diameter at most five has a shellable $r$-independence complex for all $r \geq 2$.

\begin{theorem}\label{tree-diam}
    If $G$ is a tree with $\diam(G) \leq 5$, then $\ind_r(G)$ is shellable for all $r \geq 2$.
\end{theorem}

\begin{proof}
By Theorem \ref{russ-result}, it is enough to prove that $\con_r(G)$ is $\ww$-chordal.
    If $\diam(G) \leq 4$, then by Theorem~\ref{block-diam}, $\con_r(G)$ is $\ww$-chordal for all $r \geq 2$. Now assume $\diam(G) = 5$. By Proposition~\ref{del-ope}, it suffices to show that every $c'$-minor of $\con_r(G)$ has a simplicial vertex when $G$ is connected.
    Let $\mathcal{H}$ be a $c'$-minor of $\con_r(G)$. Since $G$ is a tree with diameter $5$, there exists a maximum-length path $x_0x_1x_2x_3x_4x_5$. Thus, $G$ is a lobster tree. Assume the notation from Set-up~\ref{allsetup}.
If $\ell = 1$, then by Lemma~\ref{first-sim}, $\mathcal{H}$ contains a simplicial vertex. Consider next the case $\ell = 2$. Note that $V_2 = A_1 \cup A_2$, where
    $A_1 = N_G(x_3) \setminus (V_1 \cup \{x_2\})$, $A_2 = N_G(x_2) \setminus (V_1 \cup \{x_3\})$.
    Let $V_2 \cap V(\mathcal{H}) = \{v_1, \ldots, v_s\}$. Using the notation from Observation~\ref{block-nota}, define
    $|S_{v_k}| = \min \left\{ |S_{v_i}| \mid i = 1, \ldots, s \right\}.$
    We claim that $v_k$ is a simplicial vertex of $\mathcal{H}$. To prove this, consider two edges $e_1, e_2 \in E(\mathcal{H})$ such that $v_k \in e_1 \cap e_2$ and $e_1 \neq e_2$. By Lemma~\ref{tech-rs}, there exist edges $f_1, f_2 \in E(\con_r(G))$ such that
 $S_{v_k} \subseteq f_i, \quad f_i \cap V(\mathcal{H}) = e_i \quad \text{for } i = 1, 2, 
        \text{and } f_1 \cup f_2 \setminus \left(S_{v_k} \cup \{v_k\}\right) \text{ is connected.}$
  Let $x_1' \in e_1 \setminus e_2$ and $x_2' \in e_2 \setminus e_1$. Since $v_k \in V_2$, we have $v_k \in A_1$ or $v_k \in A_2$. Without loss of generality, assume $v_k \in A_1$.

    We claim that either $x_1' \in V_2$ or $x_2' \in V_2$. Suppose not. Since $V_t = \emptyset$ for all $t \geq 4$ and $\ell = 2$, it follows that $x_1', x_2' \in V_3 = \{x_2, x_3\}$. Without loss of generality, let $x_1' = x_2$ and $x_2' = x_3$. Then $v_k x_3 (= x_2') x_2 (= x_1')$ is a unique path from $v_k$ to $x_1'$. Since both $v_k, x_1' \in f_1$, and $x_3 = x_2' \in f_1$, we contradict the assumption that $x_2' \in e_2 \setminus e_1$. Hence, either $x_1' \in V_2$ or $x_2' \in V_2$.
Assume $x_1' \in V_2$ and let $x_1' = v_j$ for some $j$. If
    $\left| f_1 \cup f_2 \setminus \left(S_{v_k} \cup \{v_k\} \right) \right| \geq r + 1,$
    then by Lemma~\ref{tech-rs1}, $v_k$ is a simplicial vertex. Otherwise, suppose
    $\left| f_1 \cup f_2 \setminus \left(S_{v_k} \cup \{v_k\} \right) \right| \leq r.$
    Define $C = S_{v_j}$. Then each component of $C$ satisfies the conditions of \eqref{cond}. Since $|C| \geq |S_{v_k}|$, Lemma~\ref{tech-rs1} implies $v_k$ is simplicial in $\mathcal{H}$.

    Now consider $\ell = 3$. Let $v \in V_3 \cap V(\mathcal{H})$. Since $V_3 = \{x_2, x_3\}$ and $G \setminus (V_1 \cup V_2) = V_3$, any such vertex $v$ lies in a unique edge $\{x_2, x_3\}$. Hence, $v$ is simplicial.
Finally, observe that $\ell \geq 4$ is not possible since $V_t = \emptyset$ for all $t \geq 4$. Thus, every $c'$-minor of $\con_r(G)$ contains a simplicial vertex, completing the proof.
\end{proof}

Let $G$ be the graph shown below, referred to as the $\T_3$-graph.  
The vertex set is $V(G) = \{x_1, x_2, x_3, x_4, x_5, x_6, x_7, x_8, x_9\}$, and the edge set is  
$$E(G) = \{\{x_1,x_2\}, \{x_2,x_3\}, \{x_3,x_4\}, \{x_3,x_5\}, \{x_4,x_5\}, \{x_5,x_6\}, \{x_6,x_7\}, \{x_4,x_8\}, \{x_8,x_9\}\}.$$

\begin{center}
\begin{tikzpicture}[scale=0.4]
\draw [line width=1pt] (4,3)-- (6,3);
\draw [line width=1pt] (6,3)-- (8,3);
\draw [line width=1pt] (8.962424242424257,4.209653679653672)-- (8,3);
\draw [line width=1pt] (8.962424242424257,4.209653679653672)-- (10,3);
\draw [line width=1pt] (8,3)-- (10,3);
\draw [line width=1pt] (10,3)-- (12,3);
\draw [line width=1pt] (12,3)-- (14,3);
\draw [line width=1pt] (8.962424242424257,4.209653679653672)-- (8.979740259740273,5.525670995670988);
\draw [line width=1pt] (8.979740259740273,5.525670995670988)-- (9.014372294372308,7.136060606060599);
\begin{scriptsize}
\draw [fill=black] (4,3) circle (2.5pt);
\draw[color=black] (4.1,3.65) node {$x_1$};
\draw [fill=black] (6,3) circle (2.5pt);
\draw[color=black] (6.0,3.7) node {$x_2$};
\draw [fill=black] (8,3) circle (2.5pt);
\draw[color=black] (7.55,3.5) node {$x_3$};
\draw [fill=black] (8.9624,4.2097) circle (2.5pt);
\draw[color=black] (9.33,4.55) node {$x_4$};
\draw [fill=black] (10,3) circle (2.5pt);
\draw[color=black] (10.1,3.6) node {$x_5$};
\draw [fill=black] (12,3) circle (2.5pt);
\draw[color=black] (11.85,3.55) node {$x_6$};
\draw [fill=black] (14,3) circle (2.5pt);
\draw[color=black] (13.95,3.6) node {$x_7$};
\draw [fill=black] (8.9797,5.5257) circle (2.5pt);
\draw[color=black] (9.45,5.7) node {$x_8$};
\draw [fill=black] (9.0144,7.1361) circle (2.5pt);
\draw[color=black] (9.5,7.2) node {$x_9$};
\end{scriptsize}
\end{tikzpicture}
%\captionof{figure}{$\T_3$-graph}
\end{center}

For $r=3$, the shellability of $\ind_3(G)$ for block graphs can be characterized by the absence of certain forbidden configurations. In particular, excluding $\T_3$ as an induced subgraph guarantees shellability.

\begin{theorem}\label{block:3}
Let $G$ be a block graph. If $G$ is $\T_3$-free, then the complex $\ind_3(G)$ is shellable.
\end{theorem}

\begin{proof}
By Theorem \ref{russ-result}, it is enough to prove that $\con_3(G)$ is $\ww$-chordal.
By Proposition~\ref{del-ope}, it suffices to show that every $c'$-minor of $\con_3(G)$ has a simplicial vertex when $G$ is connected. Let $\mathcal{H}$ be a $c'$-minor of $\con_3(G)$, following the notation from Set-up~\ref{allsetup}.
If $1 \leq n \leq 4$, then by Theorem~\ref{block-diam}, $\mathcal{H}$ has a simplicial vertex. Now assume $n \geq 5$. Suppose there exists a vertex $y \in V(B_3) \setminus \{x_2, x_3\}$ such that $d(z, y) = 2$, i.e., there is a path $yaz$ for some $z \in V(G)$. Clearly $a \in V(B_3)$, if $a \notin V(B_3)$, then $G$ contains a $\T_3$, contradicting the assumption that $G$ is $\T_3$-free.
If $\ell = 1$, then by Lemma~\ref{first-sim}, $\mathcal{H}$ contains a simplicial vertex. We now consider the case $\ell \neq 1$.

 Suppose $(N_G(x_2) \setminus N_G(x_3)) \cap V_2 \cap V(\mathcal{H}) \neq \emptyset.$
This forces $\ell = 2$. Let
$(N_G(x_2) \setminus N_G(x_3)) \cap V_2 \cap V(\mathcal{H}) = \{v_1, \ldots, v_q\}.$
Using the notation from Observation~\ref{block-nota}, define
$|S_{v_k}| = \min \left\{ |S_{v_i}| \mid i \in \{1, \ldots, q\} \right\}.$
We claim that $v_k$ is a simplicial vertex of $\mathcal{H}$. To see this, let $e_1, e_2 \in E(\mathcal{H})$ be such that $v_k \in e_1 \cap e_2$ and $e_1 \neq e_2$. By Lemma~\ref{tech-rs}, there exist edges $f_1, f_2 \in E(\con_3(G))$ such that $S_{v_k} \subseteq f_i$ and $f_i \cap V(\mathcal{H}) = e_i$ for $i = 1, 2$. Moreover, the set
$f_1 \cup f_2 \setminus \left(S_{v_k} \cup \{v_k\}\right)$
is connected. Let $x_1' \in e_1 \setminus e_2$ and $x_2' \in e_2 \setminus e_1$. Note that $|S_{v_k}| \leq 2$; otherwise, $\{v_k\} \in E(\mathcal{H})$, contradicting the assumption that $\mathcal{H}$ is a $c'$-minor.

\vskip 1mm
\noindent
\textit{Case 1:} $|S_{v_k}| = 1$. Let $S_{v_k} = \{z_0\}$. We first claim that either $x_1'$ or $x_2'$ must belong to $V_2$. Suppose, for contradiction, that neither $x_1'$ nor $x_2'$ is in $V_2$.  Since $|f_i| = 4$ and $S_{v_k} \subseteq f_i$, it follows that $x_1', x_2' \in \{x_2, x_3\}$. As $x_1' \neq x_2'$, we may assume without loss of generality that $x_1' = x_2$ and $x_2' = x_3$. Since $v_k, x_3 \in f_2$, there exists a shortest path $P(v_k, x_3)$ in $f_2$, with all internal vertices in $f_2$. Given $z_0 \in f_2$ and $|f_2| = 4$, and using the fact that $v_k \in (N_G(x_2) \setminus N_G(x_3)) \cap V_2 \cap V(\mathcal{H})$, we conclude that
$P(v_k, x_3):= v_k x_2 x_3.$
This is a contradiction since $x_2 = x_1' \in f_2$ but $x_1' \in e_1 \setminus e_2$. Therefore, either $x_1' \in V_2$ or $x_2' \in V_2$. Without loss of generality, assume $x_1' \in V_2$, and let $C = S_{x_1'}$. Then each component of $C$ satisfies the conditions in \eqref{cond}. Since $|C| \geq |S_{v_k}|$, Lemma~\ref{tech-rs1} implies that $v_k$ is simplicial.

\vskip 1mm
\noindent
\textit{Case 2:} $|S_{v_k}| = 2$. We claim that at least one of $x_1'$ or $x_2'$ lies in $\{v_1, \ldots, v_q\}$. Suppose not. Then $x_1', x_2' \notin \{v_1, \ldots, v_q\}$, which would force $x_1' = x_2' = x_2$, a contradiction. Hence, without loss of generality, assume $x_1' = v_j$ for some $j$, and let $C = S_{v_j}$. Then each component of $C$ satisfies conditions (i)--(iv) in \eqref{cond}. Furthermore,
$|f_1 \cup f_2 \setminus (S_{v_k} \cup \{v_k\})| = 2 < 3.$
So by Lemma~\ref{tech-rs1}, $v_k$ is simplicial.

 Suppose $(N_G(x_2) \setminus N_G(x_3)) \cap V_2 \cap V(\mathcal{H}) = \emptyset.$
 
 \medskip
\noindent
Assume 
$(N_G(x_3) \setminus N_G(x_4)) \cap V_2 \cap V(\mathcal{H}) \neq \emptyset.$
This implies $\ell = 2$. Let 
$(N_G(x_3) \setminus N_G(x_4)) \cap V_2 \cap V(\mathcal{H}) = \{v_1', \ldots, v_p'\}.$
Using Observation~\ref{block-nota}, choose $v_t'$ such that
$|S_{v_t'}| = \min \{ |S_{v_j'}| \mid 1 \leq j \leq p \}.$
Take two distinct edges $e_1, e_2 \in E(\mathcal{H})$ with $v_t' \in e_1 \cap e_2$. By Lemma~\ref{tech-rs}, there exist edges $f_1, f_2 \in E(\con_3(G))$ such that $S_{v_t'} \subseteq f_i$ for each $i$, $f_i \cap V(\mathcal{H}) = e_i$, and $f_1 \cup f_2 \setminus (S_{v_t'} \cup \{v_t'\})$ is connected.
Let $x_1' \in e_1 \setminus e_2$ and $x_2' \in e_2 \setminus e_1$. Note $|S_{v_t'}| \leq 2$; otherwise $\{v_t'\} \in E(\mathcal{H})$, contradicting $\mathcal{H}$ being a $c'$-minor.
Suppose there exists another maximal clique path 
$\cp(A_3,A_2,A_1,B_4,\ldots,B_n)$ where each $A_i$ is a clique ($1 \leq i \leq 3$) with connecting vertices $b,a,x_3,\ldots,x_{n-1}$.
If $(N_G(a)\setminus N_G(x_3)) \cap V_2 \cap V(\mathcal{H}) \neq \emptyset$, then the conditions of previous case are satisfied, and consequently there exists a simplicial vertex in $\mathcal{H}$.
Otherwise ($(N_G(a) \setminus N_G(x_3)) \cap V_2 \cap V(\mathcal{H}) = \emptyset$), we have $b \notin V(\mathcal{H})$.

\noindent
\textit{Case 1:} $|S_{v_t'}| = 1$.
Either $x_1'$ or $x_2'$ belongs to $V_2 \cup \{x_2,a\}$. Without loss of generality, assume $x_1' \in V_2 \cup \{x_2, a\}$. Define
\[
C = \begin{cases} 
    S_{x_1'} & \text{if } x_1' \in V_2, \\
    \{x_1\} & \text{if } x_1' = x_2, \\
    \{b\} & \text{if } x_1' = a.
\end{cases}
\]
Then each component of $C$ satisfies \eqref{cond} and $|C| \geq |S_{v_t'}|$, so by Lemma~\ref{tech-rs1}, $v_t'$ is simplicial.

\medskip
\noindent
\textit{Case 2:} $|S_{v_t'}| = 2$.
Similarly, either $x_1'$ or $x_2'$ belongs to $\{v_1', \ldots, v_p'\} \cup \{x_2\}$. Let $x_0 \in V(B_1) \cap V_1$ and note that
$|f_1 \cup f_2 \setminus (S_{v_t'} \cup \{v_t'\})| = 2 < 4.$
Assuming $x_1' \in \{v_1', \ldots, v_p'\} \cup \{x_2\}$, define
\[
C = \begin{cases} 
    S_{x_j'} & \text{if } x_1' = v_j', \\
    \{x_0, x_1\} & \text{if } x_1' = x_2.
\end{cases}
\]
Again each component of $C$ satisfies \eqref{cond} and $|C| \geq |S_{v_t'}|$, proving $v_t'$ is simplicial via Lemma~\ref{tech-rs1}.
If no such maximal clique path $\cp(A_3,A_2,A_1,B_4,\ldots,B_n)$ exists, analogous arguments show $v_t'$ remains simplicial in $\mathcal{H}$

 Suppose 
$(N_G(x_3) \setminus N_G(x_4)) \cap V_2 \cap V(\mathcal{H}) = \emptyset.$
Observe that $x_0, x_1 \notin V(\mathcal{H})$, where $x_0 \in V(B_1) \cap V_1$. By Lemma~\ref{not-edge}, it follows that $x_2 \in V(\mathcal{H})$. Since $x_3 \in N_G(x_2) \setminus \{x_1\}$, Lemma~\ref{not-edge} further implies that $x_3 \in V(\mathcal{H})$. 
Given that 
$(N_G(x_2) \setminus N_G(x_3)) \cap V_2 \cap V(\mathcal{H}) = \emptyset,$
we conclude that the only edge containing $x_2$ is $\{x_2, x_3\}$, since 
$\{x_0, x_1, x_2, x_3\} \cap V(\mathcal{H}) = \{x_2, x_3\}.$
Thus, $x_2$ is a simplicial vertex of $\mathcal{H}$.
\end{proof}

As an immediate consequence of Theorem~\ref{block:3}, we obtain the following result:

\begin{corollary}\label{3:tree}  
If $G$ is a forest, then the complex $\ind_3(G)$ is shellable.  
\end{corollary}

We now introduce a special class of block graphs defined in terms of their clique-path structure.  

\begin{definition}\label{def:T1-graph}
A connected graph $H$ is called a \emph{$\mathcal{T}_1$-graph} if it is a block graph and contains a clique-path subgraph 
$P = \cp(B_1, B_2, \dots, B_n)$,  $n \geq 1$,
of maximum length, with connecting vertices $x_i \in B_i \cap B_{i+1}$ for $1 \leq i \leq n-1$, such that every clique $K$ of $H$ not belonging to $P$ intersects $P$ in exactly one of these connecting vertices, that is,
$V(K) \cap V(P) = \{x_i\}$
for some $i$ with $1 \leq i \leq n-1$. A graph $G$ is called a \emph{$\mathcal{T}_1$-graph} if each of its connected components is a $\mathcal{T}_1$-graph.
\end{definition}

We next extend this notion by allowing a broader attachment of cliques beyond the clique-path, leading to the following definition.  

\begin{definition}\label{def:T2-graph}
A connected graph $H$ is called a \emph{$\mathcal{T}_2$-graph} if it is a block graph and contains a clique-path subgraph 
$P = \cp(B_1, B_2, \dots, B_n)$, $n \geq 1$,
of maximum length, with connecting vertices $x_i \in B_i \cap B_{i+1}$ for $1 \leq i \leq n-1$, such that every clique $K$ of $H$ not belonging to $P$ satisfies one of the following conditions:
\begin{enumerate}
   \item $V(K) \cap V(P) = \{x_i\}$ for some $i$ with $1 \leq i \leq n-1$, or
   \item $V(K) \cap V(P) = \emptyset$ and there exists a clique $K'$ of $H$ satisfying (a) such that $V(K) \cap V(K') = \{y\}$ for some vertex $y$.
\end{enumerate}
A graph $G$ is called a \emph{$\mathcal{T}_2$-graph} if each of its connected components is a $\mathcal{T}_2$-graph.
\end{definition}

The following observation illustrates that both $\mathcal{T}_1$-graphs and $\mathcal{T}_2$-graphs form a hierarchy.

\begin{obs}\label{obs-cater}
Let $G$ be a connected $\mathcal{T}_1$-graph, and let $\cp(B_1,\dots,B_n)$, $n \geq 1$, be a maximum-size clique path subgraph of $G$. Since $G$ is a block graph, every vertex is either simplicial or a cut vertex. If $x$ is simplicial, then $G \setminus x$ remains connected. When $x \in B_1$ or $x \in B_n$, if $\deg(x)=1$ we may choose another maximum-size clique path, while if $\deg(x)>1$ the same clique path $\cp(B_1,\dots,B_n)$ remains valid; if $x \notin B_1 \cup B_n$, the path $\cp(B_1,\dots,B_n)$ still satisfies all conditions of Definition~\ref{def:T1-graph}. If $x$ is a cut vertex, then $G \setminus x$ decomposes into connected components, each containing a maximum-size clique path that satisfies Definition~\ref{def:T1-graph}. Thus the structure of a $\mathcal{T}_1$-graph is preserved, and hence $\mathcal{T}_1$-graphs form a hierarchy. A similar argument shows that $\mathcal{T}_2$-graphs also form a hierarchy.
\end{obs}
The structural restrictions imposed by $\TT_2$-graphs allow us to extend shellability results to higher independence complexes. In particular, we have the following.  
\begin{theorem}\label{block-4}
    If $G$ is a $\TT_2$-graph, then $\ind_4(G)$ is shellable.
\end{theorem}

\begin{proof}
By Theorem \ref{russ-result}, Proposition~\ref{del-ope} and Observation \ref{obs-cater}, it suffices to show that every $c'$-minor of $\con_4(G)$ has a simplicial vertex, assuming $G$ is connected. Let $\mathcal{H}$ be a $c'$-minor of $\con_4(G)$, and let us follow the notation introduced in Set-up~\ref{allsetup}. 
If $1 \leq n \leq 4$, then by Theorem~\ref{block-diam}, the hypergraph $\mathcal{H}$ contains a simplicial vertex. Now suppose $n \geq 5$. If $\ell = 1$, then Lemma~\ref{first-sim} guarantees that $\mathcal{H}$ has a simplicial vertex. It remains to consider the case when $\ell \neq 1$.

\vskip 1mm
\noindent
\textit{Case 1:} Suppose $N_G(x_2) \cap V_2 \cap V(\mathcal{H}) \neq \emptyset$. This forces $\ell = 2$. Let 
$N_G(x_2) \cap V_2 \cap V(\mathcal{H}) = \{v_1, \ldots, v_q\}.$
Using the notation from Observation~\ref{block-nota}, define $v_k$ such that
$|S_{v_k}| = \min \left\{ |S_{v_i}| \mid i \in \{1, \ldots, q\} \right\}.$
We claim that $v_k$ is a simplicial vertex of $\mathcal{H}$. To establish this, consider any two edges $e_1, e_2 \in E(\mathcal{H})$ with $v_k \in e_1 \cap e_2$ and $e_1 \neq e_2$. By Lemma~\ref{tech-rs}, there exist edges $f_1, f_2 \in E(\con_4(G))$ such that $S_{v_k} \subseteq f_i$ and $f_i \cap V(\mathcal{H}) = e_i$ for $i = 1, 2$. Moreover, the set
$f_1 \cup f_2 \setminus \left( S_{v_k} \cup \{v_k\} \right)$
is connected. 
Let $x_1' \in e_1 \setminus e_2$ and $x_2' \in e_2 \setminus e_1$. Note that $|S_{v_k}| \leq 3$; otherwise, if $|S_{v_k}| \geq 4$, then $\{v_k\} \in E(\mathcal{H})$, contradicting the assumption that $\mathcal{H}$ is a $c'$-minor.
Now, following the analysis in Case 1 of Theorem~\ref{block:3}, we observe:
If $|S_{v_k}| = 1$, then either $x_1'$ or $x_2'$ must belong to $V_2$.
If $2 \leq |S_{v_k}| \leq 3$, then either $x_1'$ or $x_2'$ must lie in $\{v_1, \ldots, v_q\}$.
Without loss of generality, assume:
\[
x_1' \in 
\begin{cases}
    V_2, & \text{if } |S_{v_k}| = 1, \\
    \{v_1, \ldots, v_q\}, & \text{if } 2 \leq |S_{v_k}| \leq 3.
\end{cases}
\]
In the second case, let $x_1' = v_j$ for some $j \in \{1, \ldots, q\}$. Define
\[
C = 
\begin{cases}
    S_{x_1'}, & \text{if } x_1' \in V_2, \\
    S_{v_j},  & \text{if } x_1' = v_j.
\end{cases}
\]
Then each component of $C$ satisfies all the conditions listed in \eqref{cond}. Since $|C| \geq |S_{v_k}|$, Lemma~\ref{tech-rs1} implies that $v_k$ is a simplicial vertex of $\mathcal{H}$.

\vskip 1mm
\noindent
\textit{Case 2:} Suppose $N_G(x_2) \cap V_2 \cap V(\mathcal{H}) = \emptyset$.  
\vskip 1mm  

\noindent  
\textit{Subcase 1:} Suppose $N_G(x_3) \cap V_2 \cap V(\mathcal{H}) \neq \emptyset$, which forces $\ell = 2$. Let  
$N_G(x_3) \cap V_2 \cap V(\mathcal{H}) = \{v_1', \ldots, v_s'\}.$  
Using the notation from Observation~\ref{block-nota}, define  
$|S_{v_k'}| = \min \left\{ |S_{v_i'}| \mid i \in \{1, \ldots, s\} \right\}.$  
We claim that $v_k'$ is a simplicial vertex of $\mathcal{H}$. To prove this, consider two distinct edges $e_1, e_2 \in E(\mathcal{H})$ with $v_k' \in e_1 \cap e_2$. By Lemma~\ref{tech-rs}, there exist edges $f_1, f_2 \in E(\con_4(G))$ such that $S_{v_k'} \subseteq f_i$ and $f_i \cap V(\mathcal{H}) = e_i$ for $i = 1, 2$. Moreover, the set  
$f_1 \cup f_2 \setminus \left( S_{v_k'} \cup \{v_k'\} \right)$  
is connected. Let $x_1' \in e_1 \setminus e_2$ and $x_2' \in e_2 \setminus e_1$.  

We analyze the three possibilities for $|S_{v_k'}|$:

\begin{enumerate}
    \item If $|S_{v_k'}| = 1$: 
    As in the proof of Theorem~\ref{block:3}, one of $x_1'$ or $x_2'$ must belong to $V_2 \cup \{x_2\}$. Without loss of generality, assume $x_1' \in V_2 \cup \{x_2\}$. Define  
    \[
    C =  
    \begin{cases} 
      S_{x_1'} & \text{if } x_1' \in V_2, \\  
      \{x_1\} & \text{if } x_1' = x_2.  
    \end{cases}
    \]

    \item If $|S_{v_k'}| = 2$:  
    A similar argument as above shows that $x_1' \in \{v_1', \ldots, v_s'\} \cup \{x_2\}$. Define  
    \[
    C =  
    \begin{cases} 
      S_{v_j'} & \text{if } x_1' = v_j', \\  
      \{x_0, x_1\} & \text{if } x_1' = x_2,  
    \end{cases}
    \]  
    where $x_0 \in V(B_1) \cap V_1$.

    \item If $|S_{v_k'}| = 3$:  
    Then either $x_1'$ or $x_2'$ belongs to $\{v_1', \ldots, v_s'\}$. Without loss of generality, assume $x_1' = v_j'$ and set $C = S_{v_j'}$.  
\end{enumerate}

In all cases, each component of $C$ satisfies conditions (i)–(iv) in \eqref{cond}. Hence, by Lemma~\ref{tech-rs1}, $v_k'$ is a simplicial vertex of $\mathcal{H}$.  
\vskip 1mm  

\noindent  
\textit{Subcase 2:} Suppose $N_G(x_3) \cap V_2 \cap V(\mathcal{H}) = \emptyset$.
First, assume $x_2 \notin V(\mathcal{H})$. Then, by Lemma~\ref{not-edge}, it follows that $x_3, x_4 \in V(\mathcal{H})$. Since
$N_G(x_2) \cap V_2 \cap V(\mathcal{H}) = \emptyset$ and $N_G(x_3) \cap V_2 \cap V(\mathcal{H}) = \emptyset,$
the only edge in $\mathcal{H}$ containing $x_3$ is $\{x_3, x_4\}$, implying that $x_3$ is a simplicial vertex in $\mathcal{H}$.

Now, assume $x_2 \in V(\mathcal{H})$. Suppose $\diam(G) > 5$. Since $\mathcal{H}$ is a $c'$-minor and hence cannot contain isolated edges, at least one of $x_3$ or $x_4$ must be in $V(\mathcal{H})$; otherwise, $\{x_2\}$ would form an isolated edge, contradicting the definition of a $c'$-minor.
Define $f = \{x_0, x_1, x_2, x_3, x_4\}$, where $x_0 \in V(B_1) \cap V_1$. The intersection $f \cap V(\mathcal{H})$ is:
\[
f \cap V(\mathcal{H}) = 
\begin{cases} 
\{x_2, x_3\}       & \text{if } x_3 \in V(\mathcal{H}),\ x_4 \notin V(\mathcal{H}), \\
\{x_2, x_4\}       & \text{if } x_3 \notin V(\mathcal{H}),\ x_4 \in V(\mathcal{H}), \\
\{x_2, x_3, x_4\} & \text{if } x_3, x_4 \in V(\mathcal{H}). 
\end{cases}
\]
Given that 
$N_G(x_2) \cap V_2 \cap V(\mathcal{H}) = \emptyset \quad \text{and} \quad N_G(x_3) \cap V_2 \cap V(\mathcal{H}) = \emptyset,$
no additional vertices from $V_2$ are adjacent to $x_2$ or $x_3$ in $\mathcal{H}$. Thus, the only edge in $\mathcal{H}$ containing $x_2$ is $f \cap V(\mathcal{H})$, making $x_2$ a simplicial vertex.

Now suppose $\diam(G) = 5$. Then $x_4 \notin V(\mathcal{H})$. Moreover, $x_3 \in V(\mathcal{H})$; otherwise, $\{x_2\} \in E(\mathcal{H})$, contradicting the $c'$-minor property. Hence, $\{x_2, x_3\} \in E(\mathcal{H})$. Since $N_G(x_i) \cap V_2 \cap V(\mathcal{H}) = \emptyset$ for $i = 2,3$, it follows that $x_2$ is a simplicial vertex of $\mathcal{H}$.

\vskip 1mm  
\noindent  
Thus, in all cases, $\mathcal{H}$ contains a simplicial vertex, completing the proof.
\end{proof}
For $\TT_1$-graphs, the structural conditions are strong enough to guarantee shellability uniformly once the independence parameter is sufficiently large. This yields the following general result.  
\begin{theorem}\label{block-5}
    If $G$ is a $\TT_1$-graph, then $\ind_r(G)$ is shellable for all $r \geq 5$.
\end{theorem}

\begin{proof}
By Theorem \ref{russ-result}, Proposition~\ref{del-ope} and Observation \ref{obs-cater}, it suffices to show that every $c'$-minor of $\con_r(G)$ contains a simplicial vertex whenever $G$ is connected and $r \geq 5$. Let $\mathcal{H}$ be such a $c'$-minor, following the notation from Set-up~\ref{allsetup}.
If $1 \leq n \leq 4$, then Theorem~\ref{block-diam} guarantees that $\mathcal{H}$ has a simplicial vertex. Now consider $n \geq 5$. If $\ell = 1$, then Lemma~\ref{first-sim} applies. 
Suppose $\ell = 2$. Let $V_2 \cap V(\mathcal{H}) = \{v_1, \ldots, v_s\}$. Define
$|S_{v_k}| = \min \left\{ |S_{v_i}| \mid i \in \{1, \ldots, s\} \right\}.$
We claim that $v_k$ is a simplicial vertex of $\mathcal{H}$. To show this, take any two distinct edges $e_1, e_2 \in E(\mathcal{H})$ with $v_k \in e_1 \cap e_2$. By Lemma~\ref{tech-rs}, there exist edges $f_1, f_2 \in E(\con_r(G))$ such that
$S_{v_k} \subseteq f_i$, $f_i \cap V(\mathcal{H}) = e_i$ for  $i = 1, 2,$
and $f_1 \cup f_2 \setminus (S_{v_k} \cup \{v_k\})$ is connected. Let $x_1' \in e_1 \setminus e_2$ and $x_2' \in e_2 \setminus e_1$. 
We first claim that at least one of $x_1', x_2'$ lies in $V_2$. Suppose not; then both lie in $\{x_2, \ldots, x_{n-2}\}$. Since
$V_2 = A_1 \cup A_2$,  where $A_1 = N_G(x_2) \setminus (V_1 \cup \{x_3\})$, $A_2 = N_G(x_{n-2}) \setminus (V_1 \cup \{x_{n-3}\})$,
assume $v_k \in A_1$, $x_1' = x_p$, $x_2' = x_q$ with $p < q$. Since $v_k, x_2' \in f_2$, there is a shortest path $P(v_k, x_2')$ in $f_2$ with all internal vertices in $f_2$. Then
$P(v_k, x_2') = v_k x_2 \cdots x_q, \quad \text{so } x_i \in f_2 \text{ for } i = 2, \ldots, q.$
Since $x_p \in \{x_2, \ldots, x_{q-1}\}$, we get $x_p = x_1' \in e_2$, contradicting $x_1' \in e_1 \setminus e_2$. Thus, at least one of $x_1', x_2'$ lies in $V_2$.
Assume $x_1' = v_j \in V_2$, and define $C = S_{v_j}$. Then each component of $C$ satisfies the conditions in \eqref{cond}. Since $|C| \geq |S_{v_j}|$, Lemma~\ref{tech-rs1} ensures that $v_j$ is simplicial.

Now consider $\ell = 3$. Then $V_3 = \{x_2, x_{n-2}\}$. Suppose $x_2 \in V_3 \cap V(\mathcal{H})$. If $x_2$ lies in two edges $e_1 \neq e_2$ of $\mathcal{H}$, pick $x_1' \in e_1 \setminus e_2$, $x_2' \in e_2 \setminus e_1$ with $x_1' = x_p$, $x_2' = x_q$ and $p < q$. Then $P(x_2, x_q) = x_2 x_3 \cdots x_q$, so $x_p \in \{x_3, \ldots, x_{q-1}\} \subseteq f_2$, contradicting the assumption. Thus, only one edge contains $x_2$.
A similar argument applies for $\ell \geq 4$, completing the proof.
\end{proof}

%\section{Construction of Shellability of Higher Independence Complexes}

\section{Shellability via Graph Constructions}
\label{modifications}

This section establishes the shellability of $r$-independence complexes for graphs built from several key operations. We first prove that attaching star-cliques to a vertex cover of any graph $H$ yields a shellable complex (Theorem~\ref{whisker}), a result we then extend to a broader class of chordal graphs (Theorem~\ref{she}). Furthermore, we demonstrate that shellability is preserved under a clique-based partitioning of a graph (Theorem~\ref{clique-whisker}). Finally, we introduce the concept of a clique cycle and prove that appending star-cliques to its structure also results in a shellable complex (Theorem~\ref{clique-cycle}). These constructive theorems significantly expand the known families of graphs with shellable higher independence complexes.

Let $H$ be a graph. A subset $S \subseteq V(H)$ is called a \emph{vertex cover} of $H$ if every edge of $H$ has at least one endpoint in $S$. Fix an integer $r \geq 1$. The graph $\CCG(H, S, r)$ is defined as the graph obtained from $H$ by attaching a star-clique graph $\SC(x)$ to each vertex $x \in S$, where $|V(\SC(x))| \geq r + 1$.

\begin{theorem}\label{whisker}
Let $G = \CCG(H, S, r)$, where $r \geq 1$. Then $\ind_r(G)$ is shellable.
\end{theorem}

\begin{proof}
By Theorem~\ref{russ-result}(1), it suffices to show that every $c$-minor of $\con_r(G)$ contains a simplicial vertex. Let $\hh$ be a $c$-minor of $\con_r(G)$. We consider two cases based on the intersection of $V(\hh)$ with the set of simplicial vertices of $G$.
Let $V_1$ denote the set of all simplicial vertices of $G$. Suppose first that $V_1 \cap V(\hh) \neq \emptyset$. If there exists a vertex $v \in V_1 \cap V(\hh)$ such that $\{v\} \notin E(\hh)$, then by Lemma~\ref{first-sim}, $v$ is a simplicial vertex of $\hh$, and we are done.
Assume now that $\{v\} \in E(\hh)$ for every $v \in V_1 \cap V(\hh)$. We claim that if $a \in S \cap V(\hh)$, then $\{a\} \in E(\hh)$. Consider such a vertex $a$.
If $(\SC(a) \setminus \{a\}) \cap V(\hh) \neq \emptyset$, then choose $c \in (\SC(a) \setminus \{a\}) \cap V(\hh)$. Since $\{c\} \in E(\hh)$ by assumption, we can construct a connected set $\{c_1, \ldots, c_r, c\}$ such that $c_i \notin V(\hh)$ and $c_i \neq a$ for all $i$. Then $\{c_1, \ldots, c_r, a\}$ is connected, implying $\{a\} \in E(\hh)$.
If $(\SC(a) \setminus \{a\}) \cap V(\hh) = \emptyset$, then since $|\SC(a)| \geq r+1$, we can form a connected set of size $r+1$ containing $a$ and $r$ elements from $\SC(a) \setminus \{a\}$. Hence, $\{a\} \in E(\hh)$.

We now claim that every edge $e \in E(\hh)$ satisfies $|e| = 1$. Suppose for contradiction that some $e \in E(\hh)$ has $|e| \geq 2$. Let $a \in e$, and choose $b \in e \setminus \{a\}$ such that $\dd(a, b) \leq \dd(a, x)$ for all $x \in e \setminus \{a\}$.
If $a$ and $b$ are adjacent in $H$, then since $S$ is a vertex cover, either $a \in S$ or $b \in S$. Without loss of generality, assume $a \in S$. Then $\{a\} \in E(\hh)$ by the earlier claim, contradicting $|e| \geq 2$.
If $a$ and $b$ are not adjacent in $H$, then $e = f \cap V(\hh)$ for some $f \in E(\con_r(G))$, where $a, b \in f$. There exists a shortest path $a y_1 \cdots y_k b$ with $k \geq 1$, $y_i \in f$ for all $i$, and $y_i \notin V(\hh)$. Since $a \notin S$ (otherwise $\{a\} \in E(\hh)$), and $y_1$ is adjacent to $a$, the vertex cover property implies $y_1 \in S$.
 Suppose $\{c'\} \in E(\hh)$ for some $c' \in \SC(y_1) \setminus \{y_1\}$. Then there exists a connected set $\{c_1', \ldots, c_r', c'\}$ with $c_i' \notin V(\hh)$ for all $i$. The set $\{c_1', \ldots, c_r', y_1, a\}$ is connected with size at least $r+2$, implying $\{a\} \in E(\hh)$, a contradiction.
Therefore, $(\SC(y_1) \setminus \{y_1\}) \cap V(\hh) = \emptyset$. But then $\SC(y_1) \cup \{a\}$ is a connected set of size at least $r+1$, again implying $\{a\} \in E(\hh)$-another contradiction.
Thus, no such edge $e$ with $|e| \geq 2$ exists, and all edges in $\hh$ are singletons. If $V_1 \cap V(\hh) = \emptyset$, then again every edge of $\hh$ is of size 1. In all cases, $\hh$ contains a simplicial vertex. This completes the proof.
\end{proof}

The following result extends Theorem~\ref{whisker} by demonstrating the shellability of $\ind_r(G)$ for a broader class of chordal graphs. Specifically, for a chordal graph $H$ and $G = \CCG(H, V(H), t)$, we show that $\ind_r(G)$ is shellable for all $r \leq 2t+1$, thus relaxing the earlier bound on $r$.

\begin{theorem}\label{she}
    Let $H$ be a chordal graph and let $G = \CCG(H, V(H), t)$ for some $t \geq 1$. If $r \leq 2t + 1$, then $\ind_r(G)$ is shellable.
\end{theorem}

\begin{proof}
By Theorem~\ref{russ-result}(1), it suffices to show that every $c$-minor of $\con_r(G)$ has a simplicial vertex. Let $\mathcal{H}$ be a $c$-minor of $\con_r(G)$, i.e., $\mathcal{H} = \con_r(G) / V_c$ for some $V_c \subseteq V(G)$. Since $G$ is chordal, we adopt the notation from Observation~\ref{chordal-nota}.

\medskip\noindent
\textit{Case 1: $V_1 \cap V(\mathcal{H}) \neq \emptyset$.}
 There exists $v \in V_1 \cap V(\mathcal{H})$ such that $\{v\} \notin E(\mathcal{H})$.  
By Lemma~\ref{first-sim}, $v$ is simplicial in $\mathcal{H}$.
Suppose $\{v\} \in E(\mathcal{H})$ for all $v \in V_1 \cap V(\mathcal{H})$.  
This implies that for every $x \in V(G) \setminus V(H)$, either $\{x\} \in E(\mathcal{H})$ or $x \in V_c$.

\noindent
\textsc{Claim 1:}
If $y \in V(\mathcal{H}) \cap V(H)$ and $\{y\} \notin E(\mathcal{H})$, then
$(\SC(y) \setminus \{y\}) \cap V(\mathcal{H}) = \emptyset.$

\noindent
Suppose, for contradiction, that there exists $c \in (\SC(y) \setminus \{y\}) \cap V(\mathcal{H})$. Since $\{c\} \in E(\mathcal{H})$, there exists an edge $f \in E(\con_r(G))$ of the form
$f = \{b_1, \ldots, b_r, c\},$
where $f \cap V(\mathcal{H}) = \{c\}$. However, $\{b_1, \ldots, b_r, y\}$ is connected in $G$, which would force $\{y\} \in E(\mathcal{H})$, a contradiction.

\noindent
\textsc{Claim 2:}
If $e \in E(\mathcal{H})$ and $|e| \neq 1$, then $e \in E(H)$.

\noindent
Let $e \in E(\mathcal{H})$ with $|e| \geq 2$, and suppose $e = f \cap V(\mathcal{H})$ for some $f \in E(\con_r(G))$.
If $|e| \geq 3$, take $a \in e$ and choose $x \in e \setminus \{a\}$ minimizing $\dd(a, x)$.
    If $a$ and $x$ were non-adjacent in $G$, there would exist a shortest path $P(a, x) = (a = y_0, y_1, \ldots, y_t = x)$ in $f$ with all $y_i \notin V(\mathcal{H})$ for $1 \leq i \leq t-1$ (since any $y_i \in V(\mathcal{H})$ would yield $\dd(a, y_i) < \dd(a, x)$).
     Considering $\mathrm{SC}(y_1)$, any $c' \in \mathrm{SC}(y_1) \setminus \{y_1\}$ with $\{c'\} \in E(\mathcal{H})$ would imply $\{a\} \in E(\mathcal{H})$, contradicting $\{a\} \notin E(\mathcal{H})$. Thus, $\mathrm{SC}(y_1) \setminus \{y_1\}$ must be disjoint from $V(\mathcal{H})$.
     The set $Q' = \mathrm{SC}(a) \cup \mathrm{SC}(y_1)$ is connected with $Q' \cap V(\mathcal{H}) = \{a\}$, forcing $\{a\} \in E(\mathcal{H})$, another contradiction.
Hence, $a$ and $x$ must be adjacent. Since $|\mathrm{SC}(a) \cup \mathrm{SC}(x)| \geq r+1$, there exists a connected set $g \subset \mathrm{SC}(a) \cup \mathrm{SC}(x)$ with $|g| = r+1$ containing $\{a, x\}$, yielding $g \cap V(\mathcal{H}) = \{a, x\} \in E(\mathcal{H})$. This contradicts $|e| \geq 3$, so we must have $e = \{a, x\} \in E(H)$.

Thus, the edge set of $\mathcal{H}$ can be written as
$E(\mathcal{H}) = E(L) \cup \{\text{singleton edges}\},$
where $L$ is an induced subgraph of $H$. Since $L$ is chordal, it has a simplicial vertex, and therefore $\mathcal{H}$ also has a simplicial vertex.

\medskip\noindent
\textit{Case 2: $V_1 \cap V(\mathcal{H}) = \emptyset$.}
In this case, $E(\mathcal{H}) = E(L') \cup \{\text{singleton edges}\}$ for some induced subgraph $L'$ of $H$. Since $L'$ is chordal, it admits a simplicial vertex, and thus so does $\mathcal{H}$.

In both cases, $\mathcal{H}$ has a simplicial vertex, proving the theorem.
\end{proof}

The following example demonstrates that the complex $\ind_r(G)$ may fail to be shellable when $r > 2t + 1$, even if $G = \CCG(H, V(H), t)$ for a chordal graph $H$ and $t \geq 1$.

\begin{example}\label{she-higher}
Let $H$ be a chordal graph with vertex set  
$V(H) = \{a_1, \ldots, a_n, b_1, \ldots, b_n, c, d\}$  
for some integer $n \geq 2$, and edge set  
\begin{align*}
    E(H) =\ &\big\{\{a_i, a_{i+1}\} \mid 1 \leq i \leq n-1\big\} \cup \big\{\{b_i, b_{i+1}\} \mid 1 \leq i \leq n-1\big\} \\
    &\cup\ \big\{\{a_1, c\}, \{a_1, d\}, \{b_1, c\}, \{b_1, d\}, \{c, d\}\big\}.
\end{align*}
Let $G = \CCG(H, V(H), t)$ be the graph obtained by attaching a star-clique graph $\SC(x)$ of size $t + 1$ to each vertex $x \in V(H)$. We claim that $\ind_r(G)$ is not shellable for $r = nt + n$.
Assume, for contradiction, that $\ind_r(G)$ is shellable. Define the set  
\[
S = \left( \bigcup_{i=1}^n \SC(a_i) \cup \bigcup_{j=1}^n \SC(b_j) \cup \SC(c) \cup \SC(d) \right) \setminus \{a_1, b_1, c,d\}.
\]  
By repeatedly applying \cite[Proposition 2.3]{ProvLouis} and \cite[Lemma 2.2(1)]{Russ11}, it follows that $\ind(\con_r(G) / S)$ is shellable.
However, one can verify that $\con_r(G) / S$ is isomorphic to the $4$-cycle $C_4$. By \cite[Theorem 10]{Wood2}, the complex $\ind(C_4)$ is not shellable, yielding a contradiction.
Hence, $\ind_r(G)$ is not shellable when $r = nt + n > 2t + 1$.
\end{example}

In \cite{CookNagel}, Cook II and Nagel introduced the notion of a clique whiskering of a graph. We generalize their construction as follows.
Let $G$ be a graph. A \emph{clique vertex-partition} of $G$ is a collection $\pi = \{W_1, \ldots, W_k\}$ of pairwise disjoint cliques (some of which may be empty) such that  
$\bigcup_{i=1}^k W_i = V(G).$
Given a graph $G$, a clique vertex-partition $\pi = \{W_1, \ldots, W_k\}$, and an integer $r \geq 1$, the \emph{$r$-clique whiskering} of $G$ with respect to $\pi$, denoted $G^\pi_r$, is defined as follows.
For each $1 \leq i \leq k$, let $t_i \geq r$ be an integer. The vertex set of $G^\pi_r$ is  
\[
V(G^\pi_r) = V(G) \cup \bigcup_{i=1}^k \{x_{i,1}, x_{i,2}, \ldots, x_{i,t_i}\},
\]
where each $x_{i,j}$ is a new vertex added to $G$. The edge set of $G^\pi_r$ is  
\[
E(G^\pi_r) = E(G) \cup \bigcup_{i=1}^k \left\{ \{a, b\} \mid a \neq b,\ a, b \in W_i \cup \{x_{i,1}, x_{i,2}, \ldots, x_{i,t_i}\} \right\}.
\]
In other words, for each clique $W_i$ in $\pi$, we introduce $t_i$ new vertices and form a clique on the vertex set $W_i \cup \{x_{i,1}, \ldots, x_{i,t_i}\}$.
%If $t_i = r$ for every $i$, we denote the resulting graph by $G^{\pi}_{\underline{r}}$  and call it the \emph{uniform $r$-clique whiskering}.
Note that a graph $G$ may admit multiple clique vertex-partitions. In particular, every graph has at least one such partition, namely the \emph{trivial partition}  
$\pi = \{\{x_1\}, \ldots, \{x_n\}\}$,  where  $V(G) = \{x_1, \ldots, x_n\}$.

\begin{theorem}\label{clique-whisker}
    Let $G$ be a graph, and let $\pi = \{W_1, \ldots, W_k\}$ be a clique vertex-partition of $G$. Then $\ind_r(G^\pi_r)$ is shellable for all $r \geq 1$
\end{theorem}

\begin{proof}
    By Theorem~\ref{russ-result}(1), it suffices to show that every $c$-minor of $\con_r(G^\pi_r)$ contains a simplicial vertex. Let $\hh$ be a $c$-minor of $\con_r(G^\pi_r)$, and let $V_1$ denote the set of all simplicial vertices in $G^\pi_r$.
First, suppose that $V_1 \cap V(\hh) \neq \emptyset$.  
    If there exists a vertex $v \in V_1 \cap V(\hh)$ such that $\{v\} \notin E(\hh)$, then by Lemma~\ref{first-sim}, $v$ is simplicial in $\hh$.
Otherwise, assume $\{v\} \in E(\hh)$ for all $v \in V_1 \cap V(\hh)$. We claim that in this case, $\{a\} \in E(\hh)$ for every vertex $a \in V(\hh)$.  
    If $a \in V_1 \cap V(\hh)$, then $\{a\} \in E(\hh)$ by assumption.  
    Now suppose $a \notin V_1$. Then $a \in V(G)$. Since $\pi = \{W_1, \ldots, W_k\}$ is a clique vertex-partition of $G$, $a \in W_i$ for some $1 \leq i \leq k$.  

    Consider two subcases:
   Suppose $\{c\} \in E(\hh)$ for some $c \in \{x_{i,1}, \ldots, x_{i,t_i}\}$. Then there exists a connected set $\{c_1, \ldots, c_r, c\}$ in $G^\pi_r$ such that $c_i \notin V(\hh)$ for all $1 \leq i \leq r$. Since $W_i \cup \{x_{i,1}, \ldots, x_{i,t_i}\}$ is a clique, it follows that $\{c_1, \ldots, c_r, a\}$ is also a connected set, so $\{a\} \in E(\hh)$.
Suppose $\{x_{i,1}, \ldots, x_{i,t_i}\} \cap V(\hh) = \emptyset$. Since $t_i \geq r$, we can select distinct vertices $x_{i,1}, \ldots, x_{i,r}$ such that $\{x_{i,1}, \ldots, x_{i,r}, a\}$ is connected in $G^\pi_r$, hence $\{a\} \in E(\hh)$.
    
    Thus, in either case, $\{a\} \in E(\hh)$ for all $a \in V(\hh)$, so $\hh$ is totally disconnected (i.e., all its edges are singleton sets).

    Now suppose $V_1 \cap V(\hh) = \emptyset$. Then every vertex of $\hh$ lies in $V(G)$. Repeating the same argument as above, we again conclude that $\hh$ is totally disconnected.

    In both cases, $\hh$ is totally disconnected, hence contains a simplicial vertex (every vertex is simplicial). This completes the proof.
\end{proof}

    We introduce the notion of a \emph{clique cycle graph}, which generalizes the classical concept of cycles.  
An \emph{$n$-clique cycle}, denoted by $\cc(B_1, \ldots, B_n)$, is a graph formed by a cyclic sequence of $n$ cliques $B_1, \ldots, B_n$, where $n \geq 3$. Each consecutive pair of cliques $B_i$ and $B_{i+1}$ (indices modulo $n$) intersects in exactly one vertex, denoted $x_i$. In particular, $B_n$ and $B_1$ share the vertex $x_n$, thereby completing the cycle. The vertices $x_1, \ldots, x_n$ are called the \emph{connecting vertices} of the $n$-clique cycle.
We define a graph $G(r)$, for a given integer $r \geq 1$, as follows: start with an $n$-clique cycle $\cc(B_1, \ldots, B_n)$ for some $n \geq 3$, and select a non-empty subset $A \subseteq \{x_1, \ldots, x_n\}$ of the connecting vertices. To each vertex $x \in A$, attach a star-clique graph $\SC(x)$, such that $|V(\SC(x))| \geq r + 1$ for at least one $x \in A$.

%Note that $G(r)$ is not necessarily a block graph, and it may fail to be chordal when $n \geq 4$.

\begin{theorem}\label{clique-cycle}
    Let $G = G(r)$ for some $r \geq 1$. Then $\ind_r(G)$ is shellable.
\end{theorem}

\begin{proof}
    By Theorem~\ref{russ-result}(1), it suffices to show that every $c$-minor of $\con_r(G)$ contains a simplicial vertex.  
    Let $\hh$ be a $c$-minor of $\con_r(G)$, i.e., $\hh = \con_r(G)/S$ for some subset $S \subseteq V(G)$. Since $G = G(r)$ is constructed from an $n$-clique cycle with attached star-clique graphs, there exists a connecting vertex $x \in \{x_1, \ldots, x_n\}$ such that $|V(\SC(x))| \geq r+1$. Let  
    $V(\SC(x)) = \{z_1, \ldots, z_m, x\},~ \text{with } m \geq r.$
We consider the following two cases:
    \vskip 1mm
    \noindent
    \textit{Case 1:} $\{z_1, \ldots, z_m\} \cap V(\hh) = \emptyset$.
    We further divide this case:
Suppose $x \in V(\hh)$. Then, since $x$ forms part of a star-clique of size at least $r+1$ and its neighbors are excluded from $\hh$, we have $\{x\} \in E(\hh)$. By Lemma \cite[Lemma 3]{Se75},
        $\hh \setminus x = (\con_r(G)/S) \setminus x = \con_r(G \setminus x) / S.$
        Observe that $G \setminus x$ is a $\TT_1$-graph. By Theorem~\ref{block-5}, the complex $\con_r(G \setminus x)$ has a simplicial vertex, and hence so does $\hh \setminus x$. Therefore, $\hh$ itself has a simplicial vertex.
       Suppose $x \notin V(\hh)$. Choose a vertex $c \in V(\hh) \cap \{x_1, \ldots, x_n\}$ such that 
        $\dd(c, z_m) \leq \dd(a, z_m)$ for all  $a \in V(\hh) \cap \{x_1, \ldots, x_n\}.$        
        That is, $c$ is the connecting vertex closest to $z_m$ among those present in $\hh$. Therefore $\{c\}\in E(\hh)$.  Repeating the same argument as in above, we conclude that $\hh$ contains a simplicial vertex.
    \vskip 1mm
    \noindent
   \textit{Case 2:} $\{z_1, \dots, z_m\} \cap V(\mathcal{H}) \neq \emptyset$.  
   Let $y \in \{z_1, \dots, z_m\} \cap V(\mathcal{H})$. If $y$ is not a singleton edge in $\mathcal{H}$, then Lemma~\ref{first-sim} implies that $\mathcal{H}$ contains a simplicial vertex.  
If $y$ is a singleton edge in $\mathcal{H}$, then by an argument analogous to \textit{Case 1}, $\mathcal{H}$ must also contain a simplicial vertex.  
 In all cases, the $c$-minor $\hh$ has a simplicial vertex. Thus, by Theorem~\ref{russ-result}(1), $\ind_r(G)$ is shellable.
\end{proof}

The following example illustrates that the conclusions of Theorems~\ref{whisker}, \ref{clique-whisker}, and \ref{clique-cycle} may fail when $\CCG(H, S, t)$, $G^\pi_t$, or $G(t)$ is considered with $t < r$.

\begin{example}\label{last-ex}
Let $H$ be the $4$-cycle on vertices $\{x_1, x_2, x_3, x_4\}$, and let $G$ be the graph obtained by attaching a whisker to each vertex of $H$. Clearly, $G$ lies in the class of graphs considered in Theorems~\ref{whisker}, \ref{clique-whisker}, and \ref{clique-cycle}.
Now, set $r = 2$. Let $K$ denote the set of newly added whisker vertices, i.e., $K = V(G) \setminus \{x_1, x_2, x_3, x_4\}$. Using \cite[Proposition~2.3]{ProvLouis} and \cite[Lemma~2.2(1)]{Russ11}, we obtain
$\link_{\ind_r(G)}(K) = \ind(\con_r(G) / K) = \ind(C_4).$
By \cite[Theorem~10]{Wood2}, the independence complex $\ind(C_4)$ is not shellable. Hence, $\ind_r(G)$ is not shellable either.
\end{example}

\vspace*{2mm}  
\noindent  
\textbf{Acknowledgments:} We express our sincere gratitude to Priyavrat Deshpande, Anurag Singh, and Russ Woodroofe for their valuable insights and clarifications, which significantly contributed to the development of this work.  
The second author was partially supported by the National Board for Higher Mathematics (NBHM) and the Science and Engineering Research Board (SERB).

\vspace*{1mm} \noindent
\textbf{Data availability statement.} Data sharing not applicable to this article as no datasets were
generated or analysed during the current study.

%\nocite*{}
%\bibliographystyle{alpha}  %% or 
%\bibliographystyle{plain}    %% ???
\bibliographystyle{abbrv}
\bibliography{refs_reg} 
\end{document}